\documentclass[10pt, a4paper]{article} 

\usepackage[utf8]{inputenc} 

\usepackage{graphicx} 
\usepackage[english]{babel}
\usepackage{booktabs} 
\usepackage{array} 
\usepackage{caption}

\usepackage{amsfonts,amsmath,amssymb,bbm,amsthm,amsbsy}
\usepackage[all]{xy}
\SelectTips{cm}{}

\usepackage{pdfpages}

\usepackage{mathtools}

\usepackage{titling}
\newcommand{\subtitle}[1]{%
  \posttitle{%
    \par\end{center}
    \begin{center}\Large#1\end{center}
    \vskip0.5em}%
}

\usepackage[pagebackref,
    ,pdfborder={0 0 0}
  ,urlcolor=black,a4paper,hypertexnames=false]{hyperref}

\usepackage{color}
\usepackage{pdfcolmk}

\makeatletter
\def\blfootnote{\xdef\@thefnmark{}\@footnotetext}
\makeatother
\date{\today%
    \protect\blfootnote{\copyright{\ N.~Heuer, C.~L\"oh 2019}. 
    This work was supported by the CRC~1085 \emph{Higher Invariants} 
    (Universit\"at Regensburg, funded by the DFG).
    \\
    MSC~2010 classification: 57N65, 57M07, 20J05}}

\usepackage[nouppercase]{scrpage2}

\automark[subsection]{section}
\usepackage{ifthen}
\deftripstyle{myheadings}
             {\ifthenelse{\isodd{\value{page}}}{\leftmark}{\pagemark}}
             {}
             {\ifthenelse{\isodd{\value{page}}}{\pagemark}{\rightmark}}%
             {}{}{}%
\long\def\forget#1{}

\def\args{\;\cdot\;}

\newtheorem{thm}{Theorem}[section]
\newtheorem{lemma}[thm]{Lemma}
\newtheorem{prop}[thm]{Proposition}
\newtheorem{corr}[thm]{Corollary}
\newtheorem{claim}[thm]{Claim}

\theoremstyle{remark}
\newtheorem{rmk}[thm]{Remark}

\theoremstyle{definition}
\newtheorem{defn}[thm]{Definition}
\newtheorem{exmp}[thm]{Example}
\newtheorem{setup}[thm]{Setup}

\theoremstyle{theorem}
\newtheorem{theorem}{Theorem}

\usepackage{tikz}
\usepackage{tikz-cd}

\usepackage{IEEEtrantools}

\newenvironment{equ*}[1]{\begin{IEEEeqnarray*}{#1}}{\end{IEEEeqnarray*}}

\usepackage[nottoc,notlof,notlot]{tocbibind} 
\usepackage[titles]{tocloft} 

\newcommand{\R}{\mathbb{R}}

\newcommand{\Q}{\mathbb{Q}}
\newcommand{\Z}{\mathbb{Z}}
\newcommand{\N}{\mathbb{N}}

\DeclareMathOperator{\scl}{scl}
\DeclareMathOperator{\cl}{cl}

\DeclareMathOperator{\rk}{rk}
\newcommand{\eurm}{\mathrm{eu}}
\def\eub#1#2{%
  {}^{#2}\eurm_b^{#1}}
\def\eu#1#2{%
  {}^{#2}\eurm^{#1}}
\newcommand{\orrm}{\mathrm{or}}
\def\orb#1#2{%
  {}^{#2}\orrm_b^{#1}}
\DeclareMathOperator{\Eurm}{Eu}

\newcommand{\Homeo}{\mathrm{Homeo}^+}

\newcommand{\sign}{\mathrm{sign}}
\newcommand{\Or}{\mathrm{Or}}
\newcommand{\gvrm}{\overline{\mathrm{gv}}}
\DeclareMathOperator{\PL}{PL}
\DeclareMathOperator{\rot}{rot}
\newcommand{\alt}{\mathrm{alt}}
\newcommand{\Arm}{\mathrm{A}}

\DeclareMathOperator{\SV}{SV}

\DeclareMathOperator{\vol}{vol}
\DeclareMathOperator{\fillc}{fill}
\DeclareMathOperator{\sfill}{sfill}
\DeclareMathOperator{\id}{id}

\DeclareMathOperator{\Hom}{Hom}
\DeclareMathOperator{\map}{map}

\newcommand{\hMCG}[1]{\mathcal{H}_{#1}}
\newcommand{\MCG}[1]{\mathcal{M}_{#1}}

\newcommand{\inj}{\hookrightarrow}

\newcommand{\col}{\colon}

\def\qand{\quad\text{and}\quad}

\def\fa#1{%
  \forall_{#1}\;\;\;}
\def\exi#1{%
  \exists_{#1}\;\;\;}

\def\longrightarrow{\to}
\def\longmapsto{\mapsto}

\def\varepsilon{\epsilon}

\title{The spectrum of simplicial volume}
\author{Nicolaus Heuer, Clara L\"oh}

\begin{document}

\maketitle

\begin{abstract}
New constructions in group homology allow us to manufacture high-dimensional manifolds with controlled simplicial volume. We prove that for every dimension bigger than~$3$ the set of simplicial volumes of orientable closed connected manifolds is dense in~$\R_{\geq 0}$. In dimension~$4$ we prove that every non-negative rational number is the simplicial volume of some orientable closed connected $4$-manifold.
Our group theoretic results relate stable commutator length to the $l^1$-semi-norm of certain singular homology classes in degree~$2$. The output of these results is translated into manifold constructions using cross-products and Thom realisation.
\end{abstract}

\section{Introduction}
The simplicial volume~$\| M \|$ of an orientable closed connected (occ) manifold~$M$ is a homotopy invariant that captures the complexity of representing fundamental classes
by singular cycles with real coefficients (see Section~\ref{sec:simvol} for a precise definition and basic terminology).
Simplicial volume is known to be positive in the
presence of enough negative
curvature~\cite{vbc,thurston,inoueyano,lafontschmidt} and known to vanish in the
presence of enough amenability~\cite{vbc,ivanov,yano,bucherconnelllafont}.
Moreover, it provides a topological lower bound for the minimal Riemannian volume (suitably normalised) in the case of smooth manifolds~\cite{vbc}.

Until now, for large dimensions~$d$, very little was known about the precise structure of the set~$\SV(d) \subset \R_{\geq 0}$ of simplicial volumes of occ $d$-manifolds. The set~$\SV(d)$ is countable and closed
under addition (Remark~\ref{rem:svsetbasics}). 
However, the set of simplicial volumes is fully understood only in dimensions $2$ and~$3$ with $\SV(2) = \N[4]$ (Example \ref{exa:gap2}) and $\SV(3)= \N[\frac{\vol(M)}{v} \mid M]$, where $M$ ranges over all complete finite-volume hyperbolic $3$-manifolds with toroidal boundary and where $v>0$ is a constant (Example \ref{exa:gap3}).

This reveals that there is a \emph{gap} of simplicial volume in dimensions $2$ and~$3$: For $d \in \{2,3\}$ there is a constant $C_d>0$ such that the simplicial volume of an occ $d$-manifold either vanishes or is at least~$C_d$.
It was an open question~\cite[p.~550]{sambusetti} whether such a gap exists in higher dimensions.
For example, until now the lowest known simplicial volume of an occ $4$-manifold has been~$24$~\cite{bucherprodsurf} (Example \ref{exa:gap4}).

In the present paper, we show that dimensions $2$ and $3$ are the \emph{only} dimensions with such a gap.
\begin{theorem}[no-gap; Section~\ref{subsec:proofA}]\label{theorem:nogap}
Let $d \geq 4$ be an integer. For every $\epsilon > 0$ there is an orientable closed connected $d$-manifold $M$ such that $\protect{0<\| M \| \leq \epsilon}$.
Hence, the set of simplicial volumes of orientable closed connected $d$-mani\-folds is dense in~$\R_{\geq 0}$.
\end{theorem}
In dimension~$4$, we get the following refinement of Theorem \ref{theorem:nogap}.
\begin{theorem}[rational realisation; Section~\ref{subsec:proofB}]\label{theorem:simvolQ}
For every~$q \in \Q_{\geq 0}$ there is an orientable closed connected   $4$-manifold~$M_q$
  with~$\mbox{$\| M_q \|=q$}$.
\end{theorem}

\subsection*{Method}

We first compute the $l^1$-semi-norm of certain integral $2$-classes in finitely presented groups
by relating these semi-norms to stable commutator length.

To formulate this connection, we recall some definitions. 
For a group $G$ and a class $\alpha \in H_d(G;\R)$, the \emph{$l^1$-semi-norm $\| \alpha \|_1$ of $\alpha$} is the semi-norm induced by the $l^1$-norm of chains in the singular chain complex of any model of~$BG$ (Section~\ref{subsec:l1-semi-norm and simvol}). The class~$\alpha$ is \emph{integral} if it lies in the image under the change of coefficients map 
map induced by~$\Z \to \R$.
  
For an element $g \in [G,G]$ in the commutator subgroup of $G$, the \emph{commutator length $\cl_G g$} of $g$ is the minimal number of commutators in $G$ needed to express~$g$ as their product. The \emph{stable commutator length (scl)} of~$g$ is the limit $\scl_G g := \lim_{n \to \infty} \cl_G(g^n) /n$. 
Stable commutator length is now well-understood for many classes of groups thanks largely to Calegari and others \cite{Calegari}. 

\begin{theorem}[Corollary \ref{cor:double}] \label{theorem:doubling}
 \sloppy Let $G$ be a finitely presented group with $H_2(G;\R) \cong 0$ and let $g \in [G,G]$ be an element
  of infinite order. Then there is a finitely presented group $D(G,g)$ and an integral class~$\alpha_g \in H_2(D(G,g);\R)$ such that
  \[ \| \alpha_g \|_1 = 8 \cdot \scl_G g.
  \]
\end{theorem}

We apply Theorem \ref{theorem:doubling} to the universal central extension $E$ of Thompson's group $T$. Recall that $T$ is the group of piecewise linear homeomorphisms of the circle with dyadic breakpoints and whose slopes are integer powers of~$2$. 
In Propostion~\ref{prop:scl on central extension of thompsons group}, we show that the universal central extension $E$ of $T$ is a finitely presented group with $H_2(E;\R) \cong 0$ and that every non-negative rational number may be realised by the stable commutator length of some element in~$E$. Using Theorem~\ref{theorem:doubling} this shows:
\begin{theorem}[Corollary~\ref{cor:controlled integral 2 classes}]\label{theorem:nogapgroup} 
For every $q \in \Q_{\geq 0}$ there is a finitely presented group $G_q$ and an integral class $\alpha_q \in H_2(G_q;\R)$ such that $\| \alpha_q \|_1 = q$. In particular, for every $\epsilon > 0$ there is a finitely presented group $G_\epsilon$ and an integral class $\alpha_\epsilon \in H_2(G_\epsilon;\R)$ such that $0 < \| \alpha_\epsilon \|_1 \leq \epsilon$.
\end{theorem}
We can now take cross-products in homology to obtain integral classes in degree greater than $3$ with crude norm control. An application of a normed version of
Thom realisation (Theorem \ref{thm:thom}) proves Theorem~\ref{theorem:nogap}.

In dimension~$4$, we refine this construction by taking products with surfaces and
using an exact computation of the product norm. This
generalises a result of Bucher~\cite{bucherprodsurf}. Theorem \ref{theorem:simvolQ} will follow from these computations.

\begin{theorem}[Corollary~\ref{corr:exact 4-classes as product with surface}]\label{theorem:prodnorm}
 Let $G$ be a group, let $\alpha \in H_2(G;\R)$, and let $\Gamma_g$ be the fundamental group of the oriented closed connected surface~$\Sigma_g$ of genus~$g \geq 2$ with fundamental class~$[\Sigma_g]_\R \in H_2(\Gamma_g;\R)$.
   Then the $l^1$-semi-norm of~$\alpha \times [\Sigma_g]_\R \in H_4(G \times \Gamma_g;\R)$ satisfies
   $$
   \bigl\| \alpha \times [\Sigma_g]_\R \bigr\|_1 =  6 \cdot (g - 1) \cdot \| \alpha \|_1.
   $$
\end{theorem}

In particular, we establish the following connection between stable commutator length and simplicial volume in dimension~$4$:

\begin{theorem}[Corollary \ref{cor:exact4}] \label{theorem:exact4mfd}
 Let $G$ be a finitely presented group that satisfies $H_2(G;\R) \cong 0$ and let $g \in [G,G]$ be an element in the commutator subgroup. Then there is an orientable closed connected $4$-manifold~$M_g$ with
  \[ \|M_g \| = 48 \cdot \scl_G g.
  \]
\end{theorem}

\subsection*{Organisation of this article}

Sections \ref{sec:simvol}, \ref{sec:bc}, and \ref{sec:scl} recall basic properties and known results on simplicial volume, bounded cohomology and stable commutator length, respectively.

In Section \ref{sec:thompson} we compute $\scl$ on the universal
central extension~$E$ of Thompson's group~$T$
(Proposition~\ref{prop:scl on central extension of thompsons group}).
This will be used in Section~\ref{sec:fillings} to construct
integral $2$-classes with controlled $l^1$-semi-norms. There we also
show Theorems~\ref{theorem:doubling} and~\ref{theorem:nogapgroup}.

In Section \ref{sec:l1 norm of prod with surfaces} we get the refinement for dimension $4$ in group homology: We compute the $l^1$-semi-norm of cross-products of general $2$-classes with certain Euler-extremal $2$-classes (Theorem \ref{thm:prodnorm}). As a corollary we obtain Theorem~\ref{theorem:prodnorm}. 

All manifolds constructed in this article will arise via a suitable version of Thom's realisation theorem in Section~\ref{sec:manufac sim vol}.
This allows us to manufacture manifolds with controlled simplicial volume and to prove Theorems~\ref{theorem:nogap}, \ref{theorem:simvolQ}, and~\ref{theorem:exact4mfd}.

A discussion of related problems may be found in Section~\ref{subsec:related problems}.

\subsection*{Acknowledgements}
The authors are grateful to the anonymous referee for carefully reading the paper and for many 
constructive suggestions. The authors would also like to thank Roberto Frigerio for pointing out an inconsistent use of the Euler cocycle in a previous version of this article.
The first author would like to thank Martin Bridson for his invaluable support and very helpful discussions.
He would further like to thank Lvzhou (Joe) Chen for many helpful discussions on stable commutator length.

\section{Simplicial volume}\label{sec:simvol}

We recall the $l^1$-semi-norm on homology and simplicial volume and
establish some notation. In particular, we collect basic properties
related to classes in degree~$2$.

\subsection{The $l^1$-semi-norm and simplicial volume} \label{subsec:l1-semi-norm and simvol}

The notion of simplicial volume of manifolds is based on the
$l^1$-semi-norm on singular homology. More precisely: Let $X$
be a topological space and let $d \in \N$. Then the \emph{$l^1$-semi-norm}
on~$H_d(X;\R)$ is 
\begin{align*}
  \| \cdot \|_1 \colon H_d(X;\R) & \longrightarrow \R_{\geq 0}
  \\
  \alpha & \longmapsto \inf \bigl\{ |c|_1 \bigm| c \in
  C_d(X;\R),\ \partial c = 0,\ [c] = \alpha \bigr\};
\end{align*}
here, $C_d(X;\R)$ is the singular chain module of~$X$ in degree~$d$
with $\R$-coefficients and $|\cdot|_1$ denotes the $l^1$-norm
on~$C_d(X;\R)$ associated with the basis of singular simplices.
More generally, if $A \subset X$ is a subspace, one can also consider
the \emph{relative $l^1$-semi-norm} on~$H_d(X,A;\R)$ induced
by the $l^1$-semi-norm on~$C_d(X;\R)$.

The $l^1$-semi-norm is a functorial semi-norm in the sense of
Gromov~\cite[p.~302]{gromov_metric}:

\begin{rmk} \label{rmk:l1functorial}
  If $f \colon X \longrightarrow Y$ is continuous, $d \in \N$,
  and $\alpha \in H_d(X;\R)$, then
  \[ \bigl\| H_d(f;\R)(\alpha) \bigr\|_1 \leq \|\alpha\|_1.
  \]
\end{rmk}

\begin{defn}[simplicial volume~\cite{vbc}]
  Let $M$ be an oriented closed connected $d$-dimensional manifold.
  Then the \emph{simplicial volume of~$M$} is defined by
  \[ \|M\| := \bigl\| [M]_\R \bigr\|_1,
  \]
  where $[M]_\R \in H_d(M;\R)$ denotes the $\R$-fundamental class
  of~$M$.
  
  More generally, if $(M,\partial M)$ is an oriented compact connected
  $d$-manifold with boundary, then one defines the \emph{relative
    simplicial volume of~$(M,\partial M)$} by
  \[ \|M,\partial M\| := \bigl\| [M,\partial M]_\R \bigr\|_1,
  \]
  where $[M,\partial M]_\R \in H_d(M,\partial M;\R)$ denotes the
  relative $\R$-fundamental class of~$(M,\partial M)$.

  Because the definition of simplicial volume is independent of the
  chosen orientation, we will also speak of the simplicial volume of
  orient\emph{able} manifolds.
\end{defn}

On the one hand, simplicial volume clearly is a topological invariant
of (orientable) compact manifolds that is compatible with mapping
degrees. On the other hand, simplicial volume is related in a
non-trivial way to Riemannian volume, e.g., in the case of hyperbolic
manifolds~\cite{vbc,thurston}. Therefore, simplicial volume
is a useful invariant in the study of rigidity properties of manifolds.

Basic examples of simplicial volumes are listed in
Examples~\ref{exa:gap2}, \ref{exa:gap3}, and~\ref{exa:gap4}. In
addition to geometric arguments, a key tool for working with
simplicial volume is bounded cohomology (see
Proposition~\ref{prop:duality simvol bc} below).

\subsection{Simplicial volume in low dimensions and gaps}

We collect the low-dimensional examples of simplicial volume as stated in the introduction.
Recall that for~$d\in \N$ we define~$\SV(d) \subset \R_{\geq 0}$ via
$$
 \SV(d) := \bigl\{ \|M\| \bigm| \text{$M$ is an orientable closed connected $d$-manifold} \bigr\}.
$$

\begin{rmk}\label{rem:svsetbasics}
  As there are only countably many homotopy types of orientable closed
  connected (occ) manifolds~\cite{mathercounting}, the set~$\SV(d)$ is
  countable for every~$d \in \N$. 
  
  The set~$\SV(d)$ is also closed
  under addition. For $d \geq 3$, this follows from the additivity of simplicial volume under connected sums~\cite{vbc}\cite[Corollary~7.7]{Frigerio} and for $d=2$ this follows from the explicit computation of $\SV(2)$ as seen in Example~\ref{exa:gap2}.
\end{rmk}

Clearly, $\SV(0) = \{1\}$ (the only relevant manifold being a single point) and
$\SV(1) = \{0\}$ (the only manifold being the circle).

\begin{exmp}[dimension~$2$]\label{exa:gap2}
For an orientable closed connected surface $\Sigma_g$ of genus $g \geq 1$ we have~$\| \Sigma_g \| = 2 \cdot \bigl|\chi(\Sigma_g)\bigr| = 4 \cdot (g-1)$ \cite{vbc,benedettipetronio}\cite[Corollary~7.5]{Frigerio}. 
Hence, 
$$
\SV(2) = \{ 0, 4, 8, \ldots \} = \N [4].
$$
We observe that the \emph{gap} in simplicial volume of dimension $2$ is $4$.
\end{exmp}

\begin{exmp}[dimension~$3$]\label{exa:gap3}
  We have~\cite{vbc,soma}\cite[Corollary~7.8]{Frigerio}
  \begin{align*}
    \SV(3) = \N
     \Bigl[ \frac{\vol(M)}{v_3}
       \Bigm| \; & \text{$M$ is a complete hyperbolic $3$-manifold}\\
                 & \text{with toroidal boundary and finite volume}
     \Bigr]
  \end{align*}
  and where $v_3$ is the maximal volume of an ideal simplex in~$\mathbb{H}^3$.
  This shows that there is a gap of simplicial volume in dimension~$3$,
  namely~$w/v_3 \approx 0.928\dots$, where $w$ is the volume of the Weeks
  manifold~\cite{gabaimeyerhoffmilley}. Moreover, the set~$\SV(3)$
  has countably many accumulation points (because the 
  set of hypbolic volumes has the order type~$\omega^\omega$~\cite{thurston}). 
\end{exmp}

\begin{exmp}[dimension~$4$]\label{exa:gap4}
The smallest known Riemannian volume $\vol(M)$ of an occ \emph{hyperbolic} $4$-manifold is~$64 \cdot
  \pi^2/3$~\cite{condermaclachlan}.
  In view of the computation of the simplicial volume of hyperbolic
  manifolds~\cite{vbc,thurston}\cite[Chapter~7.3]{Frigerio} this 
means that the smallest known simplicial volume of a hyperbolic occ $4$-manifold is $\frac{64 \cdot \pi^2}{3 \cdot v_4} \in [700,800]$ where $v_4$ is the maximal volume of an ideal $4$-simplex in $\mathbb{H}^4$.

If $\Sigma_g$, $\Sigma_h$ are orientable closed
  connected surfaces of genus~$g,h \geq 1$,
  respectively, then Bucher \cite{bucherprodsurf} showed that $\|\Sigma_g \times \Sigma_h\| = \frac32 \cdot \|\Sigma_g\| \cdot \|\Sigma_h\|$.
Hence, $\| \Sigma_2 \times \Sigma_2 \| = 24$.
This has been the smallest known non-trival simplicial volume of a $4$-manifold.
More general surface bundles over surfaces do not yield lower estimates~\cite{hosterkotschick,bucherbundles}. Also the recent
  computations/estimates for mapping tori in
  dimension~$4$~\cite{bucherneofytidis} do not produce improved gap
  bounds.
\end{exmp}

\subsection{The $l^1$-semi-norm in degree~$2$}

As classes in degree~$2$ will play an important role in our
constructions, we collect some basic properties concerning the
$l^1$-semi-norm in degree~$2$.

\begin{prop}[$l^1$-semi-norm in degree~$2$;
  \protect{\cite[Proposition~2.4]{crowleyloeh}}] \label{prop:l1 of 2-classes}
  Let $X$ be a topological space and let $\alpha \in H_2(X;\R)$. Then
  \begin{align*}
    \| \alpha \|_1
      = \inf \biggl\{ \sum_{j=1}^k |a_j| \cdot \| \Sigma_{(j)}\|
      \biggm| \; & k \in \N, a_1,\dots, a_k \in \R\setminus\{0\},
      \\
      & \text{$\Sigma_{(1)}, \dots, \Sigma_{(k)}$ orientable closed connected surfaces,}
      \\
      & \text{$f_1 \colon \Sigma_{(1)} \rightarrow X, \dots, f_k \colon \Sigma_{(k)} \rightarrow X$ continuous}
      \\
      & \text{with $\sum_{j=1}^k a_j \cdot H_2(f_j;\R)[\Sigma_{(j)}]_\R = \alpha$ in~$H_2(X;\R)$}
             \biggr\}.
  \end{align*}
\end{prop}

\begin{rmk}
  Let $X$ be a path-connected topological space, let $\alpha \in
  H_2(X;\Z)$, and let $\alpha_\R \in H_2(X;\R)$ be the image
  of~$\alpha$ under the change of coefficients map $\Z \to \R$. Then
  the description of~$\|\alpha_\R\|_1$ from Proposition~\ref{prop:l1
    of 2-classes} simplifies as follows: We have
  \[
  \| \alpha_\R \|_1 = \inf_{(f,\Sigma) \in \Sigma(\alpha)} \frac{\| \Sigma \| }{|n(f,\Sigma)|}
     = \inf_{(f,\Sigma) \in \Sigma(\alpha)} \frac{4 \cdot \bigl(g(\Sigma) - 1)}{|n(f,\Sigma)|}, 
  \]
  where $\Sigma(\alpha)$ is the class of all pairs~$(f,\Sigma)$ consisting of
  an orientable closed connected surface~$\Sigma$ of genus~$g(\Sigma) \geq 1$
  and a continuous map~$f \colon \Sigma \longrightarrow X$ with~$H_2(f;\Z)[\Sigma] =
  n(f,\Sigma) \cdot \alpha$ in~$H_2(X;\Z)$ for some integer~$n(f,\Sigma) \in \Z$.
\end{rmk}

In Section~\ref{sec:fillings}, we will relate $l^1$-semi-norms of
relative classes in degree~$2$ to filling invariants and stable
commutator length.

\subsection{Simplicial volume of products} \label{subsec:simvol of prod} 

We recall basic results on $l^1$-semi-norms of homological
cross-products.

\begin{prop}\label{prop:l1productgeneric}
  Let $X$, $Y$ be topological spaces, let $m,n \in \N$, and
  let $\alpha \in H_m(X;\R)$, $\beta \in H_n(Y;\R)$. Then
  the cross-product~$\alpha \times \beta \in H_{m+n}(X\times Y;\R)$
  satisfies
  \[ \| \alpha \|_1 \cdot \|\beta\|_1 \leq \| \alpha \times \beta\|_1
     \leq {{m+n} \choose m} \cdot \|\alpha\|_1 \cdot \|\beta\|_1.
  \]
\end{prop}
\begin{proof}
  The lower estimate follows from the duality principle
  (Proposition~\ref{prop:duality simvol bc}) and an explicit
  description of the cohomological cross-product (in bounded
  cohomology), the upper estimate follows from an explicit description
  of the homological
  cross-product~\cite{vbc}\cite[Theorem~F.2.5]{benedettipetronio}
  (this classical argument works also for general homology classes,
  not only for fundamental classes of manifolds).
\end{proof}

However, in general, it seems to be a hard problem to compute the
exact values of~$l^1$-semi-norms of products. One of the few known
cases are products of two orientable closed connected surfaces, whose
simplicial volumes have been computed by Bucher:

\begin{thm}[\protect{\cite[Corollary~3]{bucherprodsurf}}] \label{thm:bucher_prod_surf}
  Let $\Sigma_g, \Sigma_h$ be orientable closed connected surfaces of genus~$g,h \in \N_{\geq 1}$. Then
  \[
  \| \Sigma_g \times \Sigma_h \| = \frac{3}{2} \cdot \| \Sigma_g \| \cdot \| \Sigma_h \|
    = 24 \cdot (g-1) \cdot (h-1).
  \]
\end{thm}

We will generalise this theorem in Section~\ref{sec:l1 norm of prod with surfaces}. 
For now, let us note that in combination with the description of the
$l^1$-semi-norm in degree~$2$ in terms of surfaces, we obtain the
following general, improved, upper bound:

\begin{corr}\label{cor:l1prod2improvedupperbound}
  Let $X$ and $Y$ be path-connected topological spaces and
  let $\alpha \in H_2(X;\R)$, $\beta \in H_2(Y;\R)$. Then
  \[ \|\alpha \times \beta \|_1
     \leq \frac32 \cdot \|\alpha\|_1 \cdot \|\beta \|_1.
  \]
\end{corr}
\begin{proof}
  We use the description of~$\|\alpha\|_1$ and $\|\beta\|_1$ from
  Proposition~\ref{prop:l1 of 2-classes}.
  Let
  \begin{align*}
    \alpha
    = \sum_{i=1}^k a_i \cdot H_2(f_i;\R)[\Sigma_{(i)}]_\R
    \qand 
    \beta
    = \sum_{j=1}^m b_j \cdot H_2(g_j;\R)[\Pi_{(j)}]_\R
  \end{align*}
  be surface presentations of~$\alpha$ and $\beta$ as in
  Proposition~\ref{prop:l1 of 2-classes}. Then
  \[ H_4(f_i \times g_j;\R) [\Sigma_{(i)} \times \Pi_{(j)}]_\R
   = H_2(f_i;\R)[\Sigma_{(i)}]_\R \times H_2(g_j;\R)[\Pi_{(j)}]_\R,
  \]
  and so
  $
    \alpha \times \beta
    =
    \sum_{i=1}^k \sum_{j=1}^m a_i \cdot b_j \cdot H_2(f_i \times g_j;\R)[\Sigma_{(i)} \times \Pi_{(j)}]_\R.
  $ 
  Therefore, by applying the triangle inequality, functoriality
  (Remark~\ref{rmk:l1functorial}), and Theorem~\ref{thm:bucher_prod_surf},
  we have 
  \begin{align*}
    \|\alpha \times \beta\|_1
    & \leq
    \sum_{i=1}^k \sum_{j=1}^m |a_i| \cdot |b_j|
    \cdot  H_2(f_i \times g_j;\R)[\Sigma_{(i)} \times \Pi_{(j)}]_\R
    \\
    & \leq
    \sum_{i=1}^k \sum_{j=1}^m |a_j| \cdot |b_j| \cdot \frac32 \cdot \|\Sigma_{(i)}\| \cdot \|\Pi_{(j)}\|
    \\
    & =
    \frac32 \cdot 
    \sum_{i=1}^k |a_j| \cdot \|\Sigma_{(i)}\|
    \cdot
    \sum_{j=1}^m |b_j| \cdot \|\Pi_{(j)}\|.
  \end{align*}
  By Proposition~\ref{prop:l1 of 2-classes}, taking the infimum over
  all such surface presentations of~$\alpha$ and $\beta$, we obtain
  $\| \alpha \times \beta \|_1
  \leq 3/2 \cdot \|\alpha\|_1 \cdot \|\beta\|_1.
  $
\end{proof}

\section{Bounded cohomology}\label{sec:bc}

Bounded cohomology of discrete
groups and topological spaces was first systematically studied by
Gromov~\cite{vbc}. Gromov established the fundamental properties of
bounded cohomology using so-called multicomplexes. Later, Ivanov
developed a more algebraic framework via
resolutions~\cite{ivanov,ivanovTNG}.  

The reference to this introduction is the recent book by Frigerio \cite{Frigerio}.
Having applications to stable commutator length and the $l^1$-semi-norm in mind we will only define bounded cohomology for trivial real and integer coefficients.

Sections \ref{subsubsec:BC Groups}, \ref{subsubsec:BC spaces} and \ref{subsec:relationship bc groups and space} discuss the (relationships between) bounded cohomology of groups and topological spaces.
In Section \ref{subsec:duality} we state the duality principle, which allows us to compute the $l^1$ semi-norm. In Section \ref{subsec:euler class} we define the Euler class.

\subsection{Bounded cohomology of groups} \label{subsubsec:BC Groups}
Let $V$ be $\R$ or $\Z$ and let $G$ be a group. We will define the bounded cohomology $H^n_b(G;V)$ of $G$ using the \emph{homogeneous resolution}.
There is also an \emph{inhomogeneous resolution}, which is useful in low dimensions. We use this resolution only in Section \ref{sec:thompson} for central extensions and refer to the literature~\cite[Chapter~1.7]{Frigerio} for the definition.

Let $C^n(G;V) := \map(G^{n+1}, V)$ be the set of set-theoretic maps from $G^{n+1}$ to $V$. 
The group~$G$ acts on $C^n(G;V)$ via
$g \cdot \phi(g_0, \ldots, g_n) = \phi(g^{-1} \cdot g_0, \ldots, g^{-1} \cdot g_n)$.
We denote by $C^n(G;V)^G$ the subset of elements in $C^n(G;V)$ that are invariant under this action.
Let $\| \cdot \|_\infty$ be the $l^\infty$-norm on $C^n(G;V)$ and let $C^n_b(G;V)$ be the corresponding subspaces of \emph{bounded} functions.

Define the simplicial coboundary maps $\delta^n \col C^n(G;V) \to C^{n+1}(G;V)$ via
$$
\delta^n(\alpha)(g_0, \ldots, g_{n+1}) := \sum_{i=0}^{n+1} (-1)^i \cdot \alpha(g_0, \ldots, \widehat{g_i}, \ldots, g_{n+1})
$$ 
where 
$\alpha(g_0, \ldots, \widehat{g_i}, \ldots, g_{n+1})$ means that the $i$-th coordinate is omitted. 
Then $\delta^n$ restricts to a map $C^n_b(G;V) \to C^{n+1}_b(G;V)$.
The cohomology of the cochain complex~$(C^{\bullet}(G;V)^G, \delta^\bullet)$ is the 
\emph{group cohomology of~$G$ with coefficients in $V$} and denoted by $H^\bullet(G;V)$. Similarly, the cohomology of the cochain complex~$(C_b^{\bullet}(G;V)^G, \delta^\bullet)$ is the 
\emph{bounded cohomology of $G$ with coefficients in $V$} and denoted by $H_b^\bullet(G;V)$.
The embedding $C_b^{\bullet}(G;\R) \inj C^{\bullet}(G;\R)$ induces a map~$c^{\bullet} \col H_b^{\bullet}(G;V) \to H^{\bullet}(G;V)$, the \emph{comparison map}.

Bounded cohomology carries additional structure, the semi-norm induced by~$\|\cdot\|_\infty$:
For an element $\alpha \in H^n_b(G;V)$, we set 
$$
\| \alpha \| := \inf \bigl\{ \| \beta \|_\infty \bigm| \beta \in C^n_b(G;V),\
  \delta^n\beta =0,\ [\beta] = \alpha \in H^n_b(G;V) \bigr\}.
$$
  Bounded cohomology is functorial in both the group and the coefficients.

\subsection{Bounded cohomology of spaces} \label{subsubsec:BC spaces}

Let $X$ be a topological space and let $S_n(X)$ be the set of singular $n$-simplices in~$X$. Moreover, let $C^n(X;V)$ be the set of maps from $S_n(X)$ to $V$. For an element $\alpha \in C^n(X;V)$ we set  
$$
\| \alpha \|_\infty := \sup \bigl\{|\alpha(\sigma)| \bigm| \sigma \in S_n(X) \bigr\} \in [0, \infty ]
$$
and let $C^n_b(X;V) \subset C^n(X;V)$ be the subset of elements that are bounded with respect to this norm.
Let $\delta^n \col C^n_b(X;V) \to C^{n+1}_b(X;V)$ be the restriction of the singular coboundary map to bounded cochains.
Then the \emph{bounded cohomology $H^{\bullet}_b(X;V)$ of $X$ with coefficients in $V$} is the cohomology of the complex~$(C^\bullet_b(X;V), \delta^\bullet)$ and denoted by~$H^{\bullet}_b(X;V)$.
For $\alpha \in H^n_b(X;V)$ we define 
$$
\| \alpha \|_{\infty} = \inf \bigl\{ \| \beta \|_\infty \bigm| \beta \in C^n_b(X;V),\ \delta^n\beta =0, [\beta] = \alpha \in H^n_b(X;V) \}
$$
and observe that $\| \cdot \|_\infty$ is a semi-norm on $H^n_b(X;V)$.
The bounded cohomology of spaces is also functorial in both spaces and coefficients.

\subsection{Relationship between bounded cohomology of groups and spaces} \label{subsec:relationship bc groups and space}

Analogously to ordinary group cohomology, bounded cohomology of groups may also be computed using classifying spaces (and thus, we will freely switch between these descriptions).
\begin{thm}[\protect{\cite[Theorem 5.5]{Frigerio}}]
Let $X$ be a model of the classifying space~$BG$ of the group~$G$.
Then $H^\bullet_b(X;\R)$ is canonically 
isometrically isomorphic to $H^\bullet_b(G;\R)$.
\end{thm}

Remarkably, this statement holds true much more generally: every topological
space with the correct fundamental group can be used to compute bounded cohomology
of groups; moreover, bounded cohomology ignores amenable kernels~\cite{vbc}:

\begin{thm}[\protect{\cite[Theorem 5.8]{Frigerio}\cite{ivanovTNG}}]
  Let $X$ be a path-connected space. Then $H^\bullet_b(X;\R)$ is
  canonically isometrically isomorphic to $H^\bullet_b(\pi_1(X);\R)$.
\end{thm}

\begin{thm}[mapping theorem~\protect{\cite[Corollary~5.11]{Frigerio}}\cite{ivanovTNG}]\label{thm:mappingtheorem}
  Let $f \colon X \longrightarrow Y$ be a continuous map between
  path-connected topological spaces.  If the induced
  homomorphism~$\pi_1(f) \colon \pi_1(X) \longrightarrow \pi_1(Y)$ is
  surjective and has amenable kernel, then $H^{\bullet}_b(f;\R) \colon H^{\bullet}_b(Y;\R)
  \longrightarrow H^{\bullet}_b(X;\R)$ is an isometric isomorphism.
\end{thm}

\subsection{Duality} \label{subsec:duality}

Bounded cohomology of groups and spaces may be used to compute the $l^1$-semi-norm of homology classes.
For what follows, let $\langle \cdot, \cdot \rangle \col H^n_b(X;V) \times H_n(X;V) \to V$ be the map given by evaluation of cochains on chains.

\begin{prop}[duality principle~\protect{\cite[Lemma~6.1]{Frigerio}}] \label{prop:duality simvol bc}
Let $X$ be a topological space and let $\alpha \in H_n(X; \R)$. Then
$$
\| \alpha \|_1  =  \sup \bigl\{ \langle \beta , \alpha \rangle \bigm| \beta \in H^n_b(X; \R),\ \| \beta \|_\infty \leq 1 \bigr\}.
$$
Moreover, the supremum is achieved.
\end{prop}

Cocycles $\beta \in C^n_b(X,\R)$ that satisfy $\| \beta \|_\infty = 1$ and $\langle [\beta], \alpha \rangle = \| \alpha \|_1$ are called \emph{extremal for~$\alpha$}.

\begin{corr}[mapping theorem for the $l^1$-semi-norm]\label{cor:l1mappingtheorem}
  Let $f\colon X \longrightarrow Y$ be a continuous map between
  path-connected topological spaces. If the induced
  homomorphism~$\pi_1(f) \colon \pi_1(X) \longrightarrow \pi_1(Y)$ is
  surjective and has amenable kernel, then $H_{\bullet}(f;\R) \colon H_{\bullet}(X;\R)
  \longrightarrow H_{\bullet}(Y;\R)$ is isometric with respect to the
  $l^1$-semi-norm.
\end{corr}
\begin{proof}
  We only need to combine the duality principle (Proposition~\ref{prop:duality simvol bc})
  with the mapping theorem in bounded cohomology (Theorem~\ref{thm:mappingtheorem}).
\end{proof}

\subsection{Alternating cochains}  \label{subsec:alternating}
Recall that $V$ denotes $\R$ or $\Z$ and
let $\alpha \in C^n_b(G;V)$ be a bounded homogeneous cochain. We say that $\alpha$ is \emph{alternating}, if for every $g_0, \ldots, g_n \in G$ and every permutation $\sigma \in S_{n+1}$ we have that
$$
\alpha(g_{\sigma(0)}, \ldots, g_{\sigma(n)}) = \sign(\sigma) \cdot \alpha(g_0, \ldots, g_n).
$$
Every (bounded) cochain $\alpha \in C^n_b(G;\R)$ has an associated \emph{alternating} (bounded) cochain $\alt^n(\alpha)$ defined via
$$
\alt^n(\alpha)(g_0, \ldots, g_n) := \frac{1}{(n+1)!} \cdot \sum_{\sigma \in S_{n+1}} \sign(\sigma) \cdot \alpha(g_{\sigma(0)}, \ldots, g_{\sigma(n)}).
$$
Observe that  $\| \alt^n(\alpha) \|_\infty \leq \| \alpha \|_{\infty}$.
The subcomplex of alternating cochains is denoted by $C^n_{b, \alt}(G;V)$. It is well-known that one can compute real bounded cohomology using alternating cochains:
\begin{prop}[\protect{\cite[Proposition 4.26]{Frigerio}}] \label{prop:alternating cocycles}
  Let $G$ be a group. 
The complex $C^\bullet_{b, \alt}(G,\R)$ isometrically computes the bounded cohomology with real coefficients.
Moreover, for every $\alpha \in C^n_b(G,\R)$ the cocycle $\alt^n_b(\alpha)$ represents the same class as $\alpha$ in~$H^n_b(G;\R)$.
\end{prop}

\subsection{Euler class and the orientation cocycle} \label{subsec:euler class}

We describe the Euler class associated to a circle action.
For details we refer to the literature~\cite{BFH, Ghys_circle}.

For three points~$x_1,x_2,x_3 \in S^1$ on the circle let $\Or(x_1,x_2,x_3) \in \{ -1, 0, 1 \}$ be the (respective) circular order.
The group $\Homeo(S^1)$ of orientation preserving homeomorphisms on the circle preserves $\Or \in C_b^2(S^1;\Z) \subset C^2(S^1;\Z)$ and satisfies a (homogeneous) cocycle condition. Hence, $\Or$ induces a (bounded) cocycle on $\Homeo(S^1)$: For~$\xi \in S^1$, the map
\[ (g_1, g_2, g_3) \mapsto \Or(g_1.\xi, g_2.\xi, g_3.\xi)
\]
is a cocycle in~$C^2_b(\Homeo(S^1);\Z)$ and the bounded cohomology
class is independent of the choice of the point~$\xi$. 
It turns out that this class is divisible by~$2$, i.e., there
is a cocycle~$\Eurm \in C_b^2(\Homeo(S^1);\Z)$, called \emph{Euler cocycle},
with
$-2 \cdot [\Eurm] = [\Or]
$
in~$H_b^2(\Homeo(S^1);\Z)$. 

\begin{rmk} \label{notation:Euler}
Let $H := \Homeo(S^1)$. 
For Euler classes (and the orientation cocycle) we will use the following notation: Capital letters ($\Eurm$) denote cocycles and lower case letters ($\eurm$) denote classes. The classes $\eurm^{\Z} \in H^2(H; \Z)$, $\eurm^{\R} \in H^2(H; \R)$, $\eub \Z {} \in H^2_b(H;\Z)$ and $\eub \R {} \in H_b^2(H;\R)$
are the ones represented by $\Eurm$ in the corresponding cohomology groups.
If a group $G$ acts on the circle by $\rho \col G \to \Homeo(S^1)$, and $\alpha$ is a class or a cocycle defined on $\Homeo(S^1)$ then
$\rho^* \alpha$ will be the pullback of $\alpha$ via $\rho$.
If $\Gamma < \Homeo(S^1)$ is a subgroup of $\Homeo(S^1)$, then we will denote the restriction of a class or a cocycle $\alpha$ to $\Gamma$ by $^{\Gamma} \alpha$. Hence, for example, $\eub \R \Gamma \in H^2_b(\Gamma; \R)$ denotes the restriction of the real bounded Euler class to~$\Gamma$.
\end{rmk}

Let $G$ be a group with a circle action $\rho \col G \to \Homeo(S^1)$.
Then $\rho^*\eurm^{\Z} \in H^2(G;\Z)$ is called the \emph{Euler class associated to the action~$\rho$}.
The Euler class induces a central extension
\[ \xymatrix{%
  1 \ar[r]
  & \Z \ar[r]
  & \widetilde{G} \ar[r]
  & G \ar[r]
  &1}
\]
of~$G$ by $\Z$, the associated \emph{Euler extension} $\widetilde{G}$. This group has the following explicit description. It is the group defined on the set $\Z \times G$ with multiplication $(z,g) \cdot (z',g') = (z+z'+ \rho^*\Eurm(1,g,g \cdot g'),g \cdot g')$. Euler extensions are useful for constructing groups with controlled stable commutator length; see Section \ref{sec:thompson}.

We note that $\rho^*\Or^{\R}$ is extremal (in the sense of Proposition~\ref{prop:duality simvol bc}) for surface groups:

\begin{exmp} \label{exmp:Euler cocycle is extremal}
Let $g \in \N_{\geq 2}$ and let $\Sigma_g$ be an oriented closed connected surface of genus~$g \in \N_{\geq 2}$. 
Recall that $\| [\Sigma_g]_\R\|_1 = \| \Sigma_g \| = 4g-4$, where $[\Sigma_g]_\R \in H_2(\Gamma_g;\Z)$ denotes the fundamental class and $\Gamma_g = \pi_1(\Sigma_g)$.
Then $\Gamma_g$ induces an action on its boundary. By identifying $\partial \Gamma_g \cong S^1$, we obtain a circle action~$\rho \col \Gamma_g \to \Homeo(S^1)$ and 
$$
\big\langle [\rho^*\Or]^\R, [\Sigma_g]_\R \bigr\rangle = 
- 2 \cdot \bigl\langle \rho^*\eub \R {}, [\Sigma_g]_\R \bigr\rangle = - 2 \cdot \chi(\Sigma_g)
= 4\cdot g-4 = \| \Sigma_g \|,
$$
i.e., $\rho^*\Or \in C^2_b(\Gamma_g; \R)$ is an extremal cocycle for the fundamental class~$[\Sigma_g]_\R$. Indeed, it is the renormalised volume cocycle of ideal simplices in~$\mathbb{H}^2$; see \cite{bucherprodsurf}.
\end{exmp}

\section{Stable commutator length}\label{sec:scl}

 In recent years the topic of stable commutator length ($\scl$) has seen a vast developemet thanks largely to Calegari and his coauthors~\cite{Calegari}. 
In this section, we will only give a brief overview of~$\scl$. 
The definition and basic properties will be given in Section~\ref{subsec:defn basic prop of scl}.
A useful tool to compute $\scl$ is Bavard's duality theorem, described in Section~\ref{subsec:bavard's duality}. We discuss examples and general properties of~$\scl$ in Section~\ref{subsec:examples scl}.

\subsection{Definition and basic properties} \label{subsec:defn basic prop of scl}
For a group $G$ let $G'$ be its commutator subgroup. The \emph{commutator length}~$\cl_G g $ of an element $g \in G'$ is defined as
$$
\cl_G g := \min \bigl\{ n \in \N \bigm| \exi{x_1,\dots,x_n,y_1,\dots, y_n \in G }g = [x_1, y_1] \cdots [x_n, y_n]
\bigr\},
$$
where $[x,y] := xyx^{-1}y^{-1}$.
The \emph{stable commutator length of $g$ in $G$} is defined as
$$
\scl_G g  := \lim_{n \to \infty} \frac{\cl_G(g^n)}{n}.
$$
If $g_1, \ldots, g_m \in G$ are such that $g_1 \cdots g_m \in G'$, we will call $g_1 + \cdots + g_m$ a \emph{chain} and define the corresponding \emph{(stable) commutator length on chains} by
\begin{align*}
  \cl_G(g_1 + \cdots + g_m)
  & := \min_{t_1, \ldots, t_m \in G} \cl_G(t_1 g_1 t_1^{-1} \cdots t_m g_m t_m^{-1}) \mbox{, and}
\\
\scl_G(g_1 + \cdots + g_m)
  & := \lim_{n \to \infty} \frac{\cl_G(g_1^n + \cdots + g_m^n)}{n}.
\end{align*}

If $\varphi \col G \to H$ is a group homomorphism, then $\scl_G g  \geq \scl_H \varphi(g)$ for all~$g \in G'$; the analogous result holds for chains. In particular, $\scl$ is invariant under automorphisms, whence under conjugation. Thus, $\scl$ on single chains agrees with the usual definition of stable commutator length.

Stable commutator length has the following geometric interpretation: If $X$ is a connected topological space and $\gamma \col S^1 \to X$ is a loop, then the stable commutator length of the associated element $[\gamma] \in \pi_1(X)$ measures the least complexity of the surface needed to bound $\gamma$ 
(we will not use this interpretation in this paper).
In Section~\ref{subsec:sclfill}, we will describe yet another interpretation of $\scl$, namely as a
topological stable-filling invariant.

\subsection{Bavard's duality theorem and bounded cohomology} \label{subsec:bavard's duality}

Let $G$ be a group. A map $\phi \col G \to \R$ is called a \emph{quasimorphism} if there is a constant $C > 0$ such that
$$
\fa{g,h \in G} \bigl|\phi(g) + \phi(h) - \phi(g\cdot h)\bigr| \leq C.
$$
The smallest such $C$ is called the \emph{defect} of $\phi$ and is denoted by $D(\phi)$.
A quasimorphism~$\phi$ is \emph{homogeneous} if in addition we have that $\phi(g^n) = n \cdot \phi(g)$ for all $g \in G$, $n \in \Z$.
Every quasimorphism $\phi \col G \to \R$ is in bounded distance to a unique homogeneous quasimorphism $\bar{\phi} \col G \to \R$, defined by setting
$$
\bar{\phi}(g) := \lim_{n \to \infty} \frac{\phi(g^n)}{n}
$$
for every $g \in G$.
Moreover, it is well-known that $D(\bar{\phi}) \leq 2 \cdot D(\phi)$~\cite[Lemma 2.58]{Calegari}.
Analogously to the duality principle (Proposition \ref{prop:duality simvol bc}) we may compute $\scl$ using homogeneous quasimorphisms:
\begin{thm}[Bavard's duality theorem~\cite{Bavard}] \label{thm:bavard's duality}
Let $G$ be a group and let $g \in G'$.
Then
$$
\scl_G g = \sup_{\phi} \frac{|\phi(g)|}{2 \cdot D(\phi)},
$$
where the supremum is taken over all homogeneous quasimorphisms $\phi \col G \to \R$. Moreover, this supremum is achieved by an \emph{extremal} quasimorphism.
\end{thm}

\begin{rmk}\label{rem:QHb}
(Homogeneous) quasimorphisms are intimately related to second bounded real cohomology. Using the inhomogeneous resolution, it can be seen that the kernel of $c^2_G \col H^2_b(G;\R) \to H^2(G;\R)$ corresponds to the space of homogeneous quasimorphisms modulo~$\Hom(G,\R)$.
  It follows then from Bavard's dualtiy theorem that the comparison map $c^2_G \col H^2_b(G;\R) \to H^2(G;\R)$ is injective if and only if $\scl_G$ vanishes on $G$.
  
  It is well known that $H^2_b(G ; \R)$ vanishes if $G$ is abelian \cite{vbc}. Thus every homogeneous quasimorphism on an abelian group is an honest homomorphism. 
\end{rmk}

\subsection{Examples} \label{subsec:examples scl}

We collect some known results for stable commutator length.

In Sections \ref{sec:fillings} and \ref{sec:manufac sim vol} we will promote $\scl$ in a finitely presented group $G$ to the simplicial volume of manifolds in higher dimension.
For this we need to assert that $H_2(G;\R)$ vanishes. Thus, we will have a particular emphasis on this condition in the examples.

\subsubsection{Vanishing}\label{exa:sclvanish}
An element $g \in G'$ may satisfy that $\scl_G g =0$ for ``trivial'' reasons, such as if $g$ is torsion or if $g$ is conjugate to its inverse.
There are many classes of groups where -- besides these trivial reasons -- stable commutator length vanishes on the whole group. Recall that this is equivalent to the injectivity of the comparison map $c^2_G \col H^2_b(G;\R) \to H^2(G;\R)$.
Examples include:
\begin{itemize}
\item amenable groups: This follows from the vanishing of $H^2_b(G;\R)$ for every amenable group~$G$ by a result of Trauber~\cite{vbc},
\item irreducible lattices in semisimple Lie groups of rank at least~$2$~\cite{BurgerMonod}, and
\item subgroups of the group~$\mathrm{PL}^+(I)$ of piecewise linear transformations of the interval~\cite{calegari_pl_inverval}.
\end{itemize}

\subsubsection{Non-abelian free groups}\label{exa:sclfree}
In contrast, Duncan and Howie \cite{DuncanHowie} showed that every element~$g \in F'\setminus\{e\}$ in the commutator subgroup of a non-abelian free group $F$ satisfies $\scl_F g \geq 1/2$. 
In a sequence of papers \cite{Calegari_rational, Calegari_sss} Calegari showed that stable commutator length is \emph{rational} in free groups and that every rational number mod $1$ is realised as the stable commutator length of some element in the free group.
Moreover, he gave an explicit, polynomial time algorithm to compute stable commutator length in free groups. This revealed a surprising distribution of those values.
We note that these results generalise to free products of cyclic groups~\cite{Walker_cyclic} and that all these groups $G$ satisfy $H_2(G;\R)\cong 0$.

\subsubsection{Gaps and groups of non-positive curvature}
A group $G$ has a \emph{gap in $\scl$} if there is a constant $C>0$ such that for every group element~$g$, we have~$\scl_G g \geq C$ unless $\scl_G g =0$ for ``trivial'' reasons such as torsion or if $g$ is conjugate to its inverse.

In the previous example, we already have seen that non-abelian free groups have a gap in stable commutator length of $1/2$. This result has recently been generalised to right-angled Artin groups \cite{Heuer}. 
Many classes of non-positively curved groups have a gap in $\scl$, though this gap may not be uniform in the whole class of groups.
Prominent examples include hyperbolic groups \cite{CalegariFujiwara}, mapping class groups \cite{BBF}, 
free products of torsion-free groups \cite{Chen} and
amalgamated free products \cite{CFL,Heuer, ChenHeuer}. 

\subsubsection{Hyperelliptic mapping class groups}\label{exa:hMCG}
  Let $g\in \N$, let $\iota \in \MCG g$ be the
  mapping class of a hyperelliptic involution of the orientable closed
  connected surface~$\Sigma_g$ of genus~$g$, and let
  \[ \hMCG g := \bigl\{ x \in \MCG g \bigm| \iota \cdot x \cdot \iota^{-1} = x \bigr\}
     \subset \MCG g
  \]
  be the \emph{hyperelliptic mapping class group of~$\Sigma_g$}.
  The group~$\hMCG g$ is finitely presented~\cite{birmanhilden} and
  satisfies~$H_2(\hMCG g;\R) \cong
  0$~\cite[Corollary~3.3]{kawazumi}\cite[Theorem~1.1]{bensoncohen}\cite{boedigheimercohenpeim}. 
  We now let~$g\geq 2$. Let $t \in \hMCG g$ be a Dehn twist about a $\iota$-invariant
  non-separating curve on~$\Sigma_g$.
  Then we have
  \[ 0 < \frac 1 {4\cdot (2 \cdot g +1)}
       \leq \scl_{\hMCG g} t
       \leq \frac 1 {2 \cdot (2 \cdot g + 3 + 1/g)};
  \]
  the first estimate is a computation by Monden~\cite[Theorem~1.2]{monden}
  (similar estimates also appear in the work of Endo and
  Kotschick~\cite[proof of Corollary~8]{endokotschick}), the second
  estimate is due to Calegari, Monden, Sato~\cite[Theorem~1.7]{calegarimondensato}.

\section{The universal central extension of Thompson's group~$T$}\label{sec:thompson} 

Thompson's group $T$ was introduced in 1965 by Richard Thompson as the first example of an infinite but finitely presented simple group.
It is the subgroup of~$\PL^+(S^1)$ which maps dyadic rationals to dyadic rationals, with dyadic breakpoints and where each derivative -- if defined -- is an integer power of~$2$ (here, we identify~$\R / \Z \cong S^1$)~\cite{CannonFloyd}.

Stable commutator length on Thompsons's group $T$ vanishes~\cite[Chapter 5]{Calegari}, but interesting values for stable commutator length arise on the central extensions of $T$ and its generalisations associated to the Euler class~\cite{zhuang} (for the definition of the Euler extension, see Section \ref{subsec:euler class}).

In this section, we extend these results about stable commutator length on these extensions to the \emph{universal} central extension~$E$ of~$T$. For a perfect group~$G$ the universal central  extension~$\widetilde{G}$ is the unique group that is a Schur covering group of~$G$. It satisfies that~$H_1(\widetilde{G};\Z) \cong 0$ and $H_2(\widetilde{G};\Z) \cong 0$ and there is an explicit construction of~$\tilde{G}$ in terms of~$G$ and~$H^2(\tilde{G};\Z)$ (which we recall during the proof of Proposition~\ref{prop:scl on central extension of thompsons group}).

\begin{prop} \label{prop:scl on central extension of thompsons group}
The universal central extension $E$ of Thompson's  group $T$ is finitely presented and satisfies that $H_1(E;\Z) \cong 0 \cong H_2(E;\Z)$.
For every non-negative rational number~$q \in \Q_{\geq 0}$, there is an element~$e_q \in E$ with~$\scl_E e_q =q$.
\end{prop}

In Section~\ref{subsec:euler of thompson}, we recall results of Zhuang~\cite{zhuang}, which describe stable commutator length on the central extension~$\widetilde T$ of~$T$ associated to the Euler class.
Using information on the (bounded) $2$-cohomology of Thompson's group~$T$ (Section~\ref{subsec:thompson group t}), we reduce stable commutator length on~$E$ to stable commutator length on~$\widetilde{T}$  and show Proposition~\ref{prop:scl on central extension of thompsons group} (Section \ref{subsec:proof of thompsons group extesnion}).

\subsection{The Euler central extension of Thompson's group~$T$} \label{subsec:euler of thompson}

We recall the connection between stable commutator length and rotation number. This connection has been established by Barge and Ghys~\cite{bargeghys} and has been used by Zhuang~\cite{zhuang} to construct finitely presented groups with transcendental stable commutator length.

\begin{thm}[\cite{bargeghys, zhuang}] \label{thm:zhuang qm on extension}
Let $\widetilde{T}$ be the central extension of Thompson's group~$T$ associated to the Euler class~$^T\eurm^\Z \in H^2(T;\Z)$. Then there is a homogeneous quasimorphism~$\rot \col \widetilde{T} \to \R$ of defect $1$, called \emph{rotation number}, that generates the space of homogeneous quasimorphisms. Hence, for all~$\widetilde t \in \widetilde T$, 
$$
\scl_{\widetilde{T}} \widetilde{t} = \frac{\bigl|\rot(\widetilde{t})\bigr|}{2}.
$$
\end{thm}
The rotation number is well studied and has a geometric meaning~\cite{BFH}. Hence, one obtains the full spectrum of stable commutator length for~$\widetilde{T}$.
\begin{corr}[\cite{Calegari}] \label{cor: every rational value is scl in thomps}
Let $\widetilde{T}$ be the central extension of Thompson's group~$T$ by the Euler class. Then the image of stable commutator length on~$\widetilde T$ is~$\Q_{\geq 0}$.
\end{corr}
\begin{proof}
Ghys and Sergiescu \cite{GS} showed that the rotation number on~$\widetilde T$ is rational. Moreover, it is well known that every rational number is realised as such a rotation number. 
To see this observe that for every integer~$n \in \N$ there is an element~$t_n \in T$ with a periodic orbit of size~$n$ that cyclically permutes the elements of this orbit. An appropriate lift~$\widetilde{t}_n$ of this element to~$\widetilde{T}$ will satisfy~$\rot \widetilde{t}_n = 1/n$. 
By taking powers of such elements we may realise every rational as a rotation number in~$\widetilde{T}$.
\end{proof}

However, Ghys and Sergiescu~\cite{GS} showed that $H_2( \widetilde{T};\Z) \cong \Z$. Thus, we cannot apply Theorem~\ref{theorem:doubling} to the group~$\widetilde{T}$.

\subsection{(Bounded) $2$-cohomology of Thompson's group~$T$} \label{subsec:thompson group t}

The cohomology of Thompson's group $T$ was computed by Ghys and Sergiescu~\cite{GS}.  
Ghys and Sergiescu showed that the $2$-cohomology $H^2(T;\Z)$ is generated by the Euler class~$^{T}\eurm^{\Z}$ (see Section~\ref{subsec:euler class}) and another class $\alpha$, which has the following combinatorial description.

For a function $\phi \col S^1 \to \R$ that admits limits on both sides at every point in~$S^1$, let $\phi(x_+)$ be the right and let $\phi(x_-)$ be the left limit at~$x\in S^1$. In this case, set $\Delta \phi(x) := \phi(x_+) - \phi(x_-)$. 
Moreover, for an element $u \in T$ we denote by $u'_r(x)$ the right derivative of $u$ at $x \in S^1$, i.e., $u'_r(x) = u'(x_+)$.

\begin{defn}[discrete Godbillon-Vey cocycle~\protect{\cite[Theorem E]{GS}}] \label{defn:gv class}
The \emph{discrete Godbillon-Vey cocycle} $\gvrm \col T \times T \to \Z$ is defined as
$$
\gvrm(u,v) := \sum_{x \in S^1} 
\begin{vmatrix}
\log_2(v)'_r & \log_2(u \circ v)'_r \\
\Delta \log_2 (v)'_r & \Delta \log_2(u \circ v)'_r
\end{vmatrix} ( x )
$$
where the (finite) sum runs over all $x \in S^1$ that are breakpoints of $v$, $u$ or~$u \circ v$.
\end{defn}
The map~$\gvrm$ is an \emph{inhomogeneous cocycle}. In this section only we will use inhomogeneous cocycles as they are better to work with in the context of central extensions; the precise definition can for instance be found in Frigerio's book~\cite[Chapter 1.7]{Frigerio}.

\begin{thm}[\protect{\cite[Corollary C, Theorem E]{GS}}] \label{thm:cohomology of thompson group}
Thompson's group $T$ satisfies that $H^2(T; \Z) \cong \Z \oplus \Z$.
Free generators are the Euler class $^{T}\eurm^{\Z}$ and a class~$\alpha$.
This class satisfies that $2 \cdot \alpha = [\gvrm] \in H^2(G;\Z)$, where $\gvrm$ is the discrete Godbillon-Vey cocycle (Definition~\ref{defn:gv class}). 
\end{thm}

For what follows we will also need to compute the \emph{bounded} cohomology of~$T$ in degree~$2$.
\begin{prop} \label{prop:alpha not bounded}
  The class $\alpha \in H^2(T;\Z)$ from Theorem~\ref{thm:cohomology of thompson group}
  cannot be represented by a bounded cocycle, i.e., $\alpha$
  is \emph{not} in the image of the comparison map $H^2_b(T;\Z) \to H^2(T;\Z)$.
  In particular, we have that 
  $$
  H^2_b(T;\R) \cong \R,
  $$
  generated by the Euler class.
\end{prop}

\begin{proof}
  Note that it is enough to show the unboundedness statement for $[\gvrm]$ as $2 \cdot \alpha = [\gvrm]$
  (Theorem~\ref{thm:cohomology of thompson group}).
We will show the proposition by evaluating $\gvrm$ on the subgroup~$\Z^2 \cong \langle a,b \rangle_T \subset T$, where $a$ and $b$ are the elements depicted in Figure \ref{fig:Generators of ab subgroup}.
\begin{figure} 
\begin{center}
\includegraphics[scale=0.8]{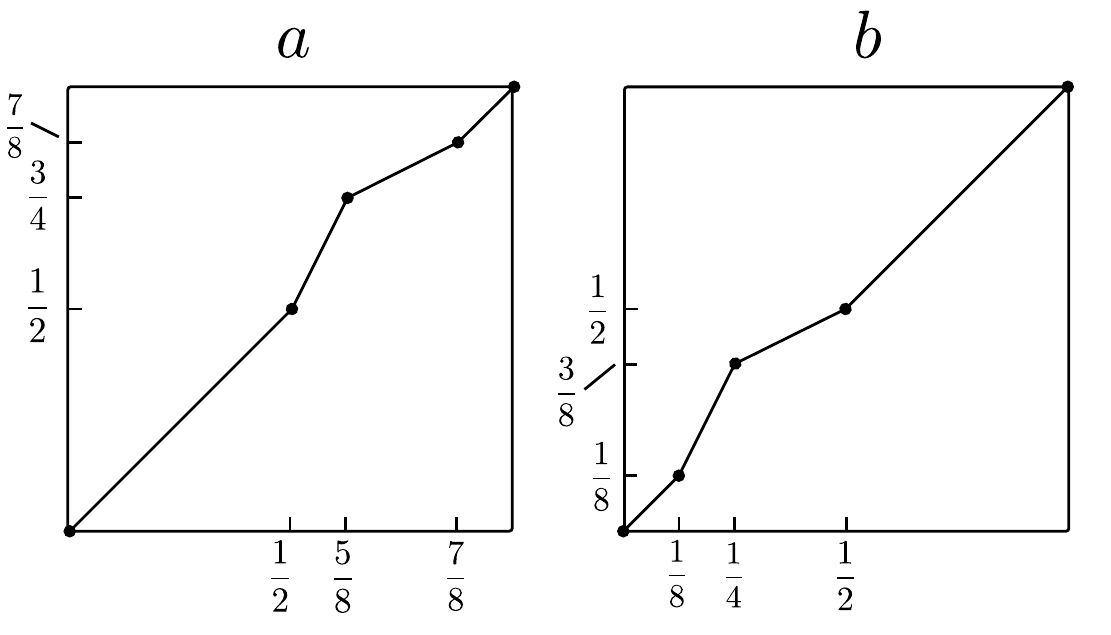}
\caption{The generators $a$ (left) and $b$ (right).} \label{fig:Generators of ab subgroup}
\end{center}
\end{figure}

\begin{claim} \label{claim:gv cocycle is gen of z2}
  The cocycle $\gvrm$ restricts on $\langle a, b \rangle_T$ to a cocycle representing
  a non-trivial element of $H^2(\Z^2;\Z) \cong \Z$.
\end{claim}

\begin{proof}[Proof of Claim \ref{claim:gv cocycle is gen of z2}]
This claim is implicitly stated in the work of Ghys and Sergiescu \cite[proof of Lemma~4.6]{GS}. For the convenience of the reader we provide an explicit proof here.

Observe that $(i,i') \mapsto \gvrm(a^i,a^{i'})$  is a (inhomogeneous) $2$-cocycle on $\Z$. Since $H^2(\Z;\Z)\cong0$, we see that there is a function $f \col \Z \to \Z$ such that $\gvrm(a^i,a^{i'})=f(i)+f(i')-f(i+i')$ for all~$i,i' \in \Z$.
Similarly, we see that there is a function $g \col \Z \to \Z$ such that $\gvrm(b^{j},b^{j'})=g(j)+g(j')-g(j+j')$.
Observe that we have 
$$
\begin{vmatrix}
\log_2(a^i)'_r & \log_2(a^{i+i'})'_r \\
\Delta \log_2 (a^i)'_r & \Delta \log_2(a^{i+i'})'_r
\end{vmatrix} ( x ) = \begin{vmatrix}
0 & 0 \\
0 & 0 
\end{vmatrix} = 0
$$
for any $x \in [0, 1/2)$ and that
$$
\begin{vmatrix}
\log_2(a^i)'_r & \log_2(a^{i+i'})'_r \\
\Delta \log_2 (a^i)'_r & \Delta \log_2(a^{i+i'})'_r
\end{vmatrix} ( 1/2 ) = \begin{vmatrix}
i & i+i' \\
i & i+i'
\end{vmatrix} = 0.
$$
This way we see that
\begin{align*}
f(i)+f(i')-f(i+i') = 
\gvrm(a^i,a^{i'})
 & =
\sum_{x \in (1/2,1]} 
\begin{vmatrix}
\log_2(a^i)'_r & \log_2(a^{i+i'})'_r \\
\Delta \log_2 (a^i)'_r & \Delta \log_2(a^{i+i'})'_r
\end{vmatrix} ( x ) \\
 &= 
\sum_{x \in (1/2,1]} 
\begin{vmatrix}
\log_2(a^i b^j)'_r & \log_2(a^{i+i'} b^{j+j'})'_r \\
\Delta \log_2 (a^i b^{j})'_r & \Delta \log_2(a^{i+i'} b^{j+j'})'_r
\end{vmatrix} ( x )
\end{align*}
for all $i,j,i',j' \in \Z$.
Similarly, we see that
\begin{align*}
  g(j)+g(j')-g(j+j') = \gvrm(b^j,b^{j'})
  &= \sum_{x \in (0,1/2)} 
\begin{vmatrix}
\log_2(b^j)'_r & \log_2(b^{j+j'})'_r \\
\Delta \log_2 (b^j)'_r & \Delta \log_2(b^{j+j'})'_r
\end{vmatrix} ( x )
\\
&= \sum_{x \in (0,1/2)} 
\begin{vmatrix}
\log_2(a^i b^j)'_r & \log_2(a^{i+i'} b^{j+j'})'_r \\
\Delta \log_2 (a^i b^{j})'_r & \Delta \log_2(a^{i+i'} b^{j+j'})'_r
\end{vmatrix} ( x ).
\end{align*}
We moreover calculate
\begin{align*}
\begin{vmatrix}
\log_2(a^i b^j)'_r & \log_2(a^{i+i'} b^{j+j'})'_r \\
\Delta \log_2 (a^i b^{j})'_r & \Delta \log_2(a^{i+i'} b^{j+j'})'_r
\end{vmatrix} (1/2) &=
\begin{vmatrix}
i & (i+i') \\
(i-j) & (i+i'-j-j')
\end{vmatrix}  \\
&=-
\begin{vmatrix}
i & i' \\
j & j'
\end{vmatrix}.
\end{align*}
Putting the above calculations together, we can now compute the restriction  of~$\gvrm$ to $\langle a, b \rangle_T$. 
For all $i,j,i',j' \in \Z$ we see that
\begin{align*}
\gvrm(a^i b^j, a^{i'} b^{j'}) 
&= 
\sum_{x \in S^1} 
\begin{vmatrix}
\log_2(a^i b^j)'_r & \log_2(a^{i+i'} b^{j+j'})'_r \\
\Delta \log_2 (a^i b^{j})'_r & \Delta \log_2(a^{i+i'} b^{j+j'})'_r
\end{vmatrix} ( x ) 
\\
&=
-\begin{vmatrix}
i & i' \\
j & j'
\end{vmatrix} +\delta^1 f_0\bigl((i,j),(i',j')\bigr) + 
\delta^1 g_0\bigl((i,j),(i',j')\bigr),
\end{align*}
where $f_0(i,j):=f(i)$ and $g_0(i,j):=g(j)$.
Hence, evaluating~$\gvrm$ on a fundamental cycle shows that
$\gvrm$ restricted to $\langle a, b \rangle_T$ represents twice a generator
of~$H^2(\Z^2;\Z) \cong \Z$. This proves Claim~\ref{claim:gv cocycle is gen of z2}.
\end{proof}

It is well-known that non-trivial elements of $H^2(\Z^2;\Z)$ cannot be represented by a bounded cocycle ($\Z^2$ is amenable). Hence,  also $[\gvrm]$ cannot be represented by a bounded cocycle, which proves the first part of Proposition \ref{prop:alpha not bounded}. 

Stable commutator length vanishes on~$T$ (Example~\ref{exa:sclvanish}) and so the comparison map $c^2_T \col H^2_b(T;\R) \to H^2(T;\R)$ is injective (Remark~\ref{rem:QHb}). 
We now assume for a contradiction that $\lambda \cdot \eu \R T + \mu \cdot \alpha \in H^2(T;\R)$ lies in the image of the comparison map~$c^2_T \col H^2_b(T;\R) \to H^2(T;\R)$ and $\mu \not = 0$.
As $\eu \R T$ is bounded and $\langle a, b \rangle_T$ is amenable, $\eu \R T$ restricts to a trivial class on $\langle a, b \rangle_T$.
Thus $\lambda \cdot \eu \R T + \mu \cdot \alpha$ restricts to $\mu \cdot \alpha$ on $\langle a, b \rangle_T$ and generates $H^2(\Z^2;\R)$ (by Claim~\ref{claim:gv cocycle is gen of z2}). This is a contradiction as these classes are not bounded. Hence, the only classes in the image of $c^2_T$ are multiples of~$\eu \R T$.
We conclude that 
$$
H^2_b(T;\R) \cong \R,
$$
generated by the Euler class. This completes the proof of Proposition~\ref{prop:alpha not bounded}.
\end{proof}

\subsection{Proof of Proposition \ref{prop:scl on central extension of thompsons group}}
\label{subsec:proof of thompsons group extesnion}

We will now prove Proposition \ref{prop:scl on central extension of thompsons group} by explicitly computing the quasimorphisms on $E$ and then invoking Bavard's duality theorem. We note that there is an alternative proof using diagrams by applying Gromov's mapping theorem. A variation of this may be found in the forthcoming paper \cite[Section 3.3]{trans_simvol}.

\begin{proof}[Proof of Proposition \ref{prop:scl on central extension of thompsons group}]
As $E$ arises as a group extension of  finitely presented groups it is itself finitely presented. The group~$T$ is simple~\cite{CannonFloyd} and thus in particular perfect.
The universal central extension~$E$ of a perfect group~$T$ always satisfies that $H_1(E;\Z)
\cong 0 \cong H_2(E;\Z)$~\cite[Chapter~6.9]{weibel}.

It remains to show that every rational number~$q \in \Q$ is the stable commutator length of some element~$e_q \in E$.
For this note that $E$ may be explicitly described as the group on the set~$\Z^2 \times T$ with multiplication
$$
\Bigl(\Big( \begin{smallmatrix}
i \\ j
\end{smallmatrix} \Big)
, t\Bigr)
\times
\Bigl(
\Big( \begin{smallmatrix}
i' \\ j'
\end{smallmatrix} \Big)
,t'\Bigr)
\mapsto  \Bigl(
\left( \begin{smallmatrix}
i+i'+^{T}\Eurm(t,t') \\ j+j'+ \Arm(t,t')
\end{smallmatrix} \right)
, t \cdot t'\Bigr),
$$
where $^{T}\Eurm$ (resp.~$\Arm$) is an inhomogeneous cocycle representing $^{T} \eurm^{\Z} \in H^2(T;\Z)$ (resp.~$\alpha \in H^2(T;\Z)$).
Similarly $\widetilde{T}$ may be described as the set $\Z \times T$ with group multiplication $(i,t) \times (i',t') \mapsto (i+i'+ ^{T}\Eurm(t,t'), t \cdot t')$.

\begin{claim} \label{claim:scl preserved}
For every $(i_0,t_0) \in \widetilde{T}$, we have that $\scl_{\widetilde{T}}(i_0,t_0) = \scl_E\bigl(
\left( \begin{smallmatrix}
i_0 \\ 0
\end{smallmatrix} \right)
, t_0\bigr)$.
\end{claim}

\begin{proof}[Proof of Claim \ref{claim:scl preserved}.]
The homomorphism $\kappa \col E \to \widetilde{T}$ defined by $\kappa \col
\bigl( \left( \begin{smallmatrix}
i \\ j
\end{smallmatrix} \right)
, t\bigr) \mapsto (i,t)$ shows by monotonicity of $\scl$ that 
$\scl_E\bigl(
\left( \begin{smallmatrix}
i_0 \\ 0
\end{smallmatrix} \right)
, t_0\bigr) \geq 
\scl_{\widetilde{T}} \kappa\bigl(
\left( \begin{smallmatrix}
i_0 \\ 0
\end{smallmatrix} \right)
, t_0\bigr) =
\scl_{\widetilde{T}} (i_0,t_0)$.
We now prove the converse inequality.

Let $\phi \col E \to \R$ be a homogeneous extremal quasimorphism to the element 
$
\bigl(\left( \begin{smallmatrix}
i_0 \\ 0
\end{smallmatrix} \right)
, t_0\bigr)
$.
Then $\phi$ restricted to the centre of~$E$ is a homogeneous quasimorphism on an abelian group, and thus a homomorphism; see Remark \ref{rem:QHb}.
 Thus there are constants~$\lambda_{{}\Eurm{}}, \lambda_{\Arm} \in \R$ such that
$\phi 
\bigl(\left( \begin{smallmatrix}
i \\ j
\end{smallmatrix} \right)
, \id \bigr) = \lambda_{\Eurm} \cdot i + \lambda_\Arm \cdot j$ for all $i,j \in \Z$.

We will first show that $\lambda_{\Arm}=0$.
For every element~$z$ in the centre and an element~$e \in E$, the group $\langle z, e \rangle$ generated by~$z$ and~$e$ is abelian and hence $\phi$ restricts to a homomorphism on $\langle z, e \rangle$ by again using Remark \ref{rem:QHb}. Hence  we have $\phi(z \cdot e) = \phi(z) + \phi(e)$ for all $z$ in the centre and $e \in E$.
We define~$\Delta \in C^2(T;\R)$ as
$
\Delta(t,t') := \delta^1 \phi
\bigl(
\bigl(
\left( \begin{smallmatrix}
0 \\ 0
\end{smallmatrix} \right)
, t\bigr)
,
\bigl(
\left(
\begin{smallmatrix}
0 \\ 0
\end{smallmatrix} \right)
, t'
\bigr)
\bigr)
$
for all~$t,t' \in T$. 
Then $\Delta$ is uniformly bounded, because $\phi$ is a quasimorphism.
We compute 
\begin{align*}
\Delta(t,t') 
&=
\phi\bigl(\left( \begin{smallmatrix}
0 \\ 0
\end{smallmatrix} \right)
, t\bigr) +
 \phi\bigl(\left( \begin{smallmatrix}
0 \\ 0
\end{smallmatrix} \right)
, t'\bigr)
-
\phi\bigl(\bigl(
\left( \begin{smallmatrix}
0 \\ 0
\end{smallmatrix} \right)
, t\bigr) \cdot
\bigl(
\left( \begin{smallmatrix}
0 \\ 0
\end{smallmatrix} \right)
, t'\bigr)
\bigr)
\\
&=
\delta^1 \psi(t,t') - \lambda_{\Eurm} \cdot ^{T}\Eurm(t,t') - \lambda_{\Arm} \Arm(t,t')
\end{align*}
for all~$t, t' \in T$ and for~$\psi \col T \to \R$ defined via~$\psi \col t \mapsto \phi(\left( \begin{smallmatrix}
0 \\ 0
\end{smallmatrix} \right)
, t)$.
Thus
$$
\lambda_{\Arm} \cdot \Arm - \delta^1 \psi  = - \lambda_{\Eurm} \cdot ^{T}\Eurm - \Delta
$$
and hence $\lambda_{\Arm} \cdot \Arm - \delta^1 \psi $ defines a bounded cocycle as the right hand side is uniformly bounded.
If $\lambda_{\Arm} \not = 0$, then this would imply that $\alpha$ may be represented by a bounded cocycle, which would contradict Proposition \ref{prop:alpha not bounded}. Thus $\lambda_{\Arm}=0$.  

Define a quasimorphism $\phi_{\widetilde{T}}$ on $\widetilde{T}$ by setting $\phi_{\widetilde{T}}(i,t) = 
\phi(\left( \begin{smallmatrix}
i \\ 0
\end{smallmatrix} \right)
, t)
$
and observe that $\phi_{\widetilde{T}}$ is homogeneous as well
and that
$
\phi_{\widetilde{T}}(i_0,t_0)
=
\phi( 
\left(
\begin{smallmatrix}
i_0 \\ 0
\end{smallmatrix} \right)
,t_0))
$.
For all $i,i',j,j' \in \Z$, $t,t' \in T$ we compute that
$$
\delta^1
\phi \Bigl( \Bigl( \Big( \begin{smallmatrix}
i \\ j
\end{smallmatrix} \Big)
, t\Bigr),
\Bigl(
\Big(  \begin{smallmatrix}
i' \\ j'
\end{smallmatrix} \Big)
, t'\Bigr)
\Bigr)
=
\delta^1 \psi(t,t') - \lambda_{\Eurm} ^{T} \Eurm(t,t')
= \delta^1 \phi_{\widetilde{T}} \bigl((i,t),(i',t') \bigr)
$$
and thus $D(\phi) = D(\phi_{\widetilde{T}})$.

Using Bavard's duality theorem we compute
$$
\scl_{\widetilde{T}}(i_0,t_0) \geq \frac{\phi_{\widetilde{T}}(i_0,t_0)}{2 D(\phi_{\widetilde{T}})}
= 
\frac{
\phi( 
\left(
\begin{smallmatrix}
i_0 \\ 0
\end{smallmatrix} \right)
,t_0)}{2 D(\phi)}
=
\scl_E(
\left(
\begin{smallmatrix}
i_0 \\ 0
\end{smallmatrix} \right)
,t_0).
$$
This proves the other inequality and thus finishes the proof of Claim \ref{claim:scl preserved}.
\end{proof}

We may now finish the proof of Proposition~\ref{prop:scl on central extension of thompsons group}. 
Every~$q \in \Q_{\geq 0}$ is the stable commutator length of some element~$t_q \in \widetilde{T}$ by Corollary \ref{cor: every rational value is scl in thomps}. Using Claim~\ref{claim:scl preserved}, we may construct an element~$e_q \in E$ with $\scl_E e_q =  \scl_{\tilde{T}} t_q = q$.
\end{proof}

\section{Fillings}\label{sec:fillings}

Stable commutator length can be interpreted as a homological filling
norm (Section~\ref{subsec:sclfill}).  After recalling the basic
notions and properties, we will use this interpretation to compute the
$l^1$-semi-norm of classes related to decomposable relators
and thus prove Theorem~\ref{theorem:doubling}
(Section~\ref{subsec:decomprel}). This will allow us to establish the
group-theoretic version of the no-gap theorem (Theorem~\ref{theorem:nogapgroup});
for the proof of theorem~\ref{theorem:doubling} and Theorem~\ref{theorem:nogapgroup},
we will only need the filling norm in dimension~$2$, as already considered
by Bavard~\cite{Bavard} and Calegari~\cite[Chapter~2.5/2.6]{Calegari}.
Moreover, we will explain how in the higher-dimensional case
the simplicial volume of manifolds can
also be viewed as a filling norm (Section~\ref{subsec:svfill}).

\subsection{Stable filling norms} 

We first recall the stable filling norm for the bar complex. We will then
extend this notion to topological spaces and higher degrees. For a group~$G$,
the bar complex~$C_{\bullet}(G;\R)$ (computing~$H_{\bullet}(G;\R)$) has the following form
in low degrees: We have~$C_1(G;\R) = \R[G]$ and $\partial_1 = 0$ as
well as~$C_2(G;\R) = \R[G]^2$ and
\begin{align*}
  \partial_2 \colon C_2(G;\R) & \longrightarrow C_1(G;\R) \\
  G \times G \ni (g,h) & \longmapsto g + h - g \cdot h.
\end{align*}
Moreover, the chain modules of~$C_{\bullet}(G;\R)$ are endowed with the $l^1$-norm
corresponding to the bar bases.

\begin{defn}[(stable) filling norm]\label{def:sfillbar}
  Let $G$ be a group. 
  \begin{itemize}
  \item If $c \in \partial C_2(G;\R)$, the \emph{filling norm of~$c$}
    is defined as
    \[ \fillc_G c := \inf \bigl\{ |b|_1 \bigm| b \in C_2(G;\R),\ \partial b = c \bigr\}.
    \]
  \item
    Let $m \in \N$ and let $r_1,\dots, r_m \in G'$.
    The \emph{stable filling norm of~$r_1 + \dots + r_m$} is defined as
    \[ \sfill_G (r_1 + \dots + r_m) := \lim_{n \rightarrow \infty}
       \frac1n \cdot \fillc_G (r_1^n + \dots + r_m^n).
    \]
  \end{itemize}
\end{defn}

Notice that the limit in the definition of the stable filling norm indeed
exists~\cite[p.~34]{Calegari}.

For the generalisation to topological spaces, we replace group
elements by loops (or maps from simplicial spheres) and we replace
taking powers of group elements by composition with self-maps of
spheres of the corresponding degree.

\begin{defn}[topological (stable) filling norms]\label{def:sfill}
  Let $d\in \N$, let $X$ be a topological space, and let $\sigma
  \colon \partial \Delta^d \longrightarrow X$ be continuous.
  \begin{itemize}
  \item If $c \in \partial(C_d(X;\R))$, the \emph{filling norm of~$c$}
    is defined as
    \[ \fillc_X c := \inf \bigl\{ |b|_1 \bigm| b \in C_d(X;\R),\ \partial b = c \bigr\}.
    \]
  \item The \emph{filling norm of~$\sigma$} is then defined as
    \[ \fillc_X \sigma := \fillc_X c_\sigma
     = \inf \bigl\{ |b|_1 \bigm| b \in C_d(X;\R),\ \partial b = c_\sigma\bigr\},
    \]
    where $c_\sigma := C_{d-1}(\sigma;\R)(\partial \id_{\Delta^d}) \in C_{d-1}(X;\R)$
    is the canonical singular cycle associated with~$\sigma$.
  \item The \emph{stable filling norm of~$\sigma$} is defined as
    \[ \sfill_X \sigma := \lim_{n \rightarrow \infty} \frac{\fillc_X \sigma[n]}n,
    \]
    where for~$n \in \N$, we write~$w_n\colon \partial \Delta^d \longrightarrow \partial \Delta^d$
    for ``the'' standard self-map of~$\partial \Delta^d \cong S^{d-1}$ of degree~$n$ and 
    $\sigma[n] := \sigma \circ w_n$.
  \item If $m \in \N$ and $\sigma_1, \dots, \sigma_m \colon \partial
    \Delta^d \longrightarrow X$ are continuous maps, then we define
    \begin{align*}
      \fillc_X (\sigma_1 + \dots + \sigma_m)
      & := \fillc_X (c_{\sigma_1} + \dots + c_{\sigma_m})
      \\
      \sfill_X (\sigma_1 + \dots + \sigma_n)
      & := \lim_{n \rightarrow \infty} \frac1n \cdot \fillc_X \bigl(\sigma_1[n] + \dots + \sigma_m[n]\bigr)
    \end{align*}
  \end{itemize}
\end{defn}

\begin{rmk}[existence of the stabilisation limit]\label{rem:sfilllim}
  The limits in the situation of the definition above indeed exist:
  For notational convenience, we only prove the existence in the case
  of~$\sfill_X \sigma$; the general case can be proved in the same way
  (with additional indices). The argument is similar to the one for
  the stable filling norm in the bar complex. The only complication is
  that, in order to compare different ``powers'', we will need to use
  the uniform boundary condition for~$C_{\bullet}(\partial\Delta^d;\R)$.

  Because $\pi_1(\partial \Delta^d)$ is amenable, there exists a
  constant~$K \in \R_{>0}$ with the following property~\cite{matsumotomorita,fauserloehvarubc}: 
  For every~$z \in
  \partial (C_d(\partial \Delta^d;\R))$ there is a~$b \in C_d(
  \partial \Delta^d;\R)$ with
  \[ z = \partial b
  \qand
    |b|_1 \leq K \cdot |z|_1.
  \]

  If~$n,m \in \N$, then the chains~$c_{w_n} + c_{w_m}$ and $c_{w_{n+m}}$
  are homologous in the complex~$C_{\bullet}(\partial \Delta^d;\R)$ (because $\deg w_n + \deg w_m
  = n + m = \deg w_{n+m}$). Thus, there exists a chain~$b_{n,m} \in C_d(\partial \Delta^d;\R)$
  such that
  \[ c_{w_{n+m}} - c_{w_n} - c_{w_m} = \partial b_{n,m}
  \qand
  |b_{n,m}|_1 \leq K \cdot 3 \cdot (d+1).
  \]
  Hence, for every continuous map~$\sigma \colon \partial \Delta^d \longrightarrow X$
  we obtain
  \begin{align*}
    c_{\sigma[n+m]} - c_{\sigma[n]} - c_{\sigma[m]}
    & = \partial C_d(\sigma;\R)(b_{n,m}) \qand
    \\
    \bigl|C_d(\sigma;\R)(b_{n,m})\bigr|_1
    & \leq |b_{n,m}|_1 \leq K \cdot 3 \cdot (d+1),
  \end{align*}
  and so
  \[ \fillc_X \sigma[n+m] \leq \fillc_X \sigma[n] + \fillc_X \sigma[m] + K \cdot 3 \cdot (d+1).
  \]
  Now elementary analysis shows that the limit~$\lim_{n\rightarrow
    \infty} 1/n \cdot \fillc_X\sigma[n]$ does exist.
\end{rmk}

\begin{rmk}[change of the self-maps]
  The map~$w_n$ is only unique up to homotopy, but homotopic choices
  for~$w_n$ lead to the same stable filling norm; this can be seen
  using the uniform boundary condition as in the proof of the
  existence of the stable filling limits
  (Remark~\ref{rem:sfilllim}). Therefore, this ambiguity will be of no
  consequence for us.
\end{rmk}
  
\begin{rmk}[change of the singular models]
  In the situation of Definition~\ref{def:sfill}, we could choose
  other singular cycle models of~$\sigma$ than~$c_\sigma$: If $c' \in
  C_{d-1}(\partial \Delta^{d};\R)$ is a fundamental cycle of~$\partial
  \Delta^d$, if $b' \in C_d(\partial \Delta^d;\R)$ satisfies~$\partial
  b' = c' - c$, and if $c'_\sigma := C_{d-1}(\sigma;\R)(c')$, then
  \begin{align*}
    \sfill_X \sigma
  & = \lim_{n \rightarrow \infty} \frac{1}n
    \cdot \inf \bigl\{ |b|_1 \bigm| b \in C_d(X;\R),\ \partial b = c_{\sigma[n]}\bigr\}
  \\  
  & = \lim_{n \rightarrow \infty} \frac{1}n
  \cdot \inf \bigl\{ \bigl|b+C_d(\sigma[n];\R))(b')\bigr|_1
  \bigm| b \in C_d(X;\R),\ \partial b = c_{\sigma[n]}\bigr\}
  \\  
  & = \lim_{n \rightarrow \infty} \frac{1}n
  \cdot \inf \bigl\{ |b|_1 \bigm| b \in C_d(X;\R),\ \partial b = c'_{\sigma[n]}\bigr\}.
  \end{align*}
  For the second equality, we use that $|C_d(\sigma[n];\R)(b')|_1 \leq |b'|_1$
  holds for all~$n \in \N$ (so that the difference in norm is negligible when
  taking~$n \rightarrow \infty$).
\end{rmk}

\begin{rmk}[bar filling vs.\ topological filling]\label{rem:fillcal}
  Let $G$ be a group, let $m \in \N$, and let $r_1, \dots, r_m \in
  G'$. If $X$ is a model of~$BG$ and $\sigma_1, \dots, \sigma_m \colon
  \partial \Delta^2 \longrightarrow X$ are loops
  representing~$r_1,\dots, r_m$, respectively, then
  \[ \sfill_G (r_1 + \dots + r_m) = \sfill_X (\sigma_1 + \dots +\sigma_m).
  \]
  
  This can be seen as follows: The standard constructions produce 
  chain maps~$\varphi \colon C_{\bullet}(G;\R) \longrightarrow C_{\bullet}(X;\R)$
  (choosing paths in~$\widetilde X$ for each group element and inductively
  filling the simplices)
  and $\psi \colon C_{\bullet}(X;\R) \longrightarrow C_{\bullet}(G;\R)$ (choosing a set-theoretic
  fundamental domain~$D$ for the deck transformation action on~$\widetilde X$
  and looking at the translates of~$D$ that contain the vertices of the lifted
  simplices) with the following properties:
  \begin{itemize}
  \item $\psi \circ\varphi = \id_{C_{\bullet}(G;\R)}$,
  \item $\varphi \circ \psi \simeq \id_{C_{\bullet}(X;\R)}$ through a chain homotopy~$h$ 
    that is bounded in every degree,
  \item $\|\varphi\| \leq 1$ and $\|\psi\| \leq 1$.
  \end{itemize}
  The first and third conditions easily imply that
  \[ \fillc_G (r_1 + \dots + r_m) \leq \fillc_X(\sigma_1 + \dots + \sigma_m);
  \]
  moreover, we have~$\psi(\sigma_j[n]) =r_j^n$ for all~$j \in \{1,\dots,m\}$
  and all~$n \in \N$. Hence, 
  \[ \sfill_G (r_1 + \dots + r_m) \leq \sfill_X (\sigma_1 + \dots + \sigma_m).
  \]
  
  Conversely, if~$b \in C_2(G;\R)$ satisfies~$\partial b = r_1 + \dots + r_m$, then
  \[ \overline b := \varphi (b) + h(\sigma_1 + \dots + \sigma_m)
  \]
  satisfies
  \begin{align*}
    \partial \overline b
    & = \partial \varphi (b) + \partial h (\sigma_1 + \dots + \sigma_m)
    \\
    & =\varphi (\partial b) + \sigma_1 + \dots + \sigma_m - \varphi \circ \psi (\sigma_1 + \dots + \sigma_m)
    \\
    & = \varphi (r_1 + \dots + r_m) + \sigma_1 + \dots + \sigma_m - \varphi(r_1 + \dots + r_m)
    \\
    & = \sigma_1 + \dots + \sigma_m.
  \end{align*}
  Thus,
  $ \fillc_X (\sigma_1 + \dots + \sigma_m )
  \leq |\overline b|_1
  \leq |b|_1 + \|h \| \cdot m.
  $
  Taking the infimum over all such~$b$ results in
  \[ \fillc_X (\sigma_1 + \dots + \sigma_m) \leq \fillc_G (r_1 + \dots + r_m) + \| h \| \cdot m.
  \]
  Passing to the stabilisation limit, we obtain
  \begin{align*}
    \sfill_X (\sigma_1 + \dots + \sigma_m)
    & = \lim_{n \rightarrow \infty} \frac1n \cdot \fillc_X \bigl(\sigma_1[n] + \dots + \sigma_m[n]\bigr)
    \\
    & \leq \lim_{n \rightarrow \infty} \frac1n \cdot \bigl( \fillc_G (r_1^n + \dots + r_m^n) + \|h\| \cdot m\bigr)
    \\
    & = \lim_{n \rightarrow \infty} \frac1n \cdot \fillc_G (r_1^n + \dots + r_m^n)
    \\
    & = \sfill_G (r_1 + \dots + r_m).
  \end{align*}
\end{rmk}

\subsection{Stable commutator length as filling invariant}\label{subsec:sclfill}

The fact that every commutator consists of four pieces has the following
generalisation in terms of filling norms:

\begin{lemma}[$\scl$ as filling invariant]\label{lem:sclfill}
  Let $G$ be a group, let $m \in \N$, and let $r_1, \dots, r_m \in G'$. Then
  \[ \scl_G (r_1 + \dots + r_m) = \frac 14 \cdot \sfill_{G} (r_1 + \dots + r_m).
  \]
\end{lemma}
\begin{proof}
  In the case of a single relator, this observation goes back to
  Bavard~\cite[Proposition~3.2]{Bavard}\cite[Lemma~2.69]{Calegari}. Calegari
  extended this equality to the case of linear
  combinations~\cite[Lemma~2.77]{Calegari}.
\end{proof}

Furthermore, as Calegari~\cite{whatisscl} puts it: ``One can interpret
stable commutator length as the infimum of the $L^1$ norm (suitably
normalized) on chains representing a certain (relative) class in group
homology.'' We will prove this statement in Corollary~\ref{cor:cylrelgen}
as a special case of the following generalisation:

\begin{prop}[relative $\l^1$-semi-norm as filling invariant]\label{prop:relsvfill}
  Let $Z$ be a CW-complex, let $m \in \N_{>0}$, $d \in \N_{\geq 2}$, let $\partial Z \subset Z$
  be a subspace that is homeomorphic to~$\coprod_m \partial\Delta^d$ and such
  that the inclusions~$\sigma_1 ,\dots, \sigma_m \colon \partial \Delta^d
  \longrightarrow Z$ of the $m$~components of~$\partial Z$ into~$Z$ are $\pi_1$-injective
  (this is automatic if~$d \geq 3$).
  \begin{enumerate}
  \item If $\beta \in H_d(Z,\partial Z;\R)$ with~$\partial \beta = [\partial Z]_\R$, then
    \[ \| \beta\|_1 \geq \sfill_Z (\sigma_1 + \dots + \sigma_m).
    \]
  \item
    If the connecting homomorphism~$\partial \colon H_d(Z,\partial Z;\R)
    \longrightarrow H_{d-1} (\partial Z ;\R)$ is an isomorphism and if
    $\beta \in H_d(Z ,\partial Z;\R)$ is the class with~$\partial
    \beta = [\partial Z]_\R$, then
    \[ \| \beta\|_1 = \sfill_Z (\sigma_1 + \dots + \sigma_m).
    \]
  \end{enumerate}
\end{prop}
\begin{proof}
  We first show the estimate~$\sfill_{Z} (\sigma_1 + \dots + \sigma_m)
  \leq \| \beta\|_1$: Clearly, the group~$\pi_1(\partial \Delta^d)$ is
  amenable; because $\sigma_1,\dots, \sigma_m$ are $\pi_1$-injective,
  the equivalence theorem~\cite{vbc}\cite[Corollary~6]{bbfipp} ensures that
  for every~$\varepsilon \in \R_{>0}$ there exists a relative
  cycle~$c \in C_d(Z;\R)$ representing~$\beta$ in~$H_d(Z,\partial Z;\R)$ with
  \begin{align}
     |c|_1 \leq \| \beta\|_1 + \varepsilon
     \qand
     |\partial c|_1 \leq \varepsilon.
     \label{eq:eqthm}
  \end{align}
  Moreover, because the fundamental group of the $m$~components
  of~$\partial Z$ are amenable, there exists a constant~$K \in
  \R_{>0}$ implementing the uniform boundary condition~\cite{matsumotomorita}: For every~$z \in
  \partial (C_d(\partial Z;\R))$ there is a~$b \in C_d(\partial
  Z;\R)$ with
  \[ z = \partial b
  \qand
    |b|_1 \leq K \cdot |z|_1.
  \]
  Let $\varepsilon \in \R_{>0}$, let $c$ be a relative cycle 
  as in~\eqref{eq:eqthm}, and let   
  $n \in \N$ with~$n \geq 1/\varepsilon$. Then
  $z := \partial c - 1/n \cdot (c_{\sigma_1[n]} + \dots + c_{\sigma_m[n]}) \in C_{d-1}(\partial Z;\R)$
  is a boundary (both summands are fundamental cycles of~$\partial Z$).
  Hence, there exists a~$b \in C_d(\partial Z;\R)$
  with
  \begin{align*}
    \partial b & = \partial c - \frac1n \cdot (c_{\sigma_1[n]} + \dots + c_{\sigma_m[n]})
    \qand
    \\
    |b|_1 & \leq K \cdot \Bigl| \partial c - \frac {1}n
    \cdot (c_{\sigma_1[n]} + \dots + c_{\sigma_m[n]})\Bigr|_1
  \leq K \cdot \Bigl(\varepsilon + m \cdot \frac{d+1}n\Bigr)
  \\ &
  \leq K \cdot \bigl(1 + m \cdot (d+1)\bigr) \cdot \varepsilon.
  \end{align*}
  The chain~$n \cdot (c - b) \in C_d(Z;\R)$ then witnesses
  that
  \begin{align*}
    \frac1 n \cdot \fillc_{Z} \bigl(\sigma_1[n] + \dots + \sigma_m[n]\bigr)
    & \leq \frac1n \cdot n \cdot |c-b|_1 = |c-b|_1
    \\
    &
    \leq \|\beta\|_1 + \varepsilon
    + K \cdot \bigl(1 + m \cdot (d+1)\bigr) \cdot \varepsilon.
  \end{align*}
  Taking first $n\rightarrow \infty$ and then~$\varepsilon \rightarrow 0$
  shows that $\sfill_{Z} (\sigma_1 + \dots + \sigma_m)
  \leq \| \beta\|_1$.
  
  Conversely, we will now prove that~$\| \beta\|_1 \leq \sfill_{Z}
  (\sigma_1 + \dots + \sigma_m)$ under the additional assumption that
  $\partial \colon H_d(Z,\partial Z;\R) \longrightarrow
  H_{d-1}(\partial Z;\R)$ is an isomorphism: Let $n \in \N_{>0}$ and
  let $b \in C_d(Z;\R)$ with~$\partial b = c_{\sigma_1[n]} + \dots +
  c_{\sigma_m[n]}$. In particular, $b$ is a relative cycle for~$(Z,
  \partial Z)$; moreover, $1/n \cdot b$ represents~$\beta$ (because
  $\partial (1/n \cdot b) = 1/n \cdot (c_{\sigma_1[n]} + \dots +
  c_{\sigma_m[n]})$ is a fundamental cycle of~$\partial Z$).  Hence,
  \[ \| \beta\|_1
     \leq \frac 1n \cdot |b|_1.
     \]
  Taking first the infimum over all such~$b$ and then~$n \rightarrow \infty$ yields 
  the desired estimate $\| \beta\|_1 \leq \sfill_{Z} (\sigma_1 + \dots + \sigma_m)$.
\end{proof}

\begin{corr}[$\scl$ as relative $l^1$-semi-norm]\label{cor:cylrelgen}
  Let $G$ be a group with~$H_2(G;\R) \cong 0$, let $m \in \N$, let $r_1,\dots, r_m \in G'$ be elements
  of infinite order, and let $X$ be a model of~$BG$.
  Let
  \begin{align*}
    Z
    & := X \cup_{r_1,\dots, r_m} \Bigl( \coprod_m S^1 \times [0,1]\Bigr)
    \\
    & := X \cup_{\coprod_{j=1}^m \gamma_j \text{ on~$\coprod_m S^1 \times\{0\}$}}
         \Bigl( \coprod_m S^1 \times [0,1] \Bigr)
  \end{align*}
  be the mapping cylinder associated with (loops~$\gamma_1,\dots, \gamma_r$ in~$X$
  representing) the elements~$r_1,\dots, r_m$, and
  let $\partial Z := \coprod_m S^1 \times \{1\} \subset Z$.
  Then there exists a unique relative homology class~$\beta \in H_2(Z,\partial Z;\R)$ 
  whose boundary class~$\partial \beta$ is the fundamental class of~$\coprod_m S^1$;
  the class~$\beta$ satisfies
  \[ \|\beta\|_1 = 4 \cdot \scl_G ( r_1 + \dots + r_m).
  \]
\end{corr}
\begin{proof}
  The long exact homology sequence of the
  pair~$(Z,\partial Z)$ shows that the connecting homomorphism~$\partial \colon H_2(Z,\partial
  Z;\R) \longrightarrow H_1(\partial Z;\R)$ is an isomorphism (by
  hypothesis, $H_2(Z;\R) \cong H_2(X;\R) \cong H_2(G;\R) \cong 0$, and the inclusion~$\partial Z
  \hookrightarrow Z$ induces the trivial homomorphism
  on~$H_1(\args;\R)$ because $r_1, \dots, r_m$ are in the commutator
  subgroup of~$G$). This shows the existence of~$\beta$.

  Because $r_1, \dots, r_m \in G'$ all have infinite order,
  the corresponding inclusions~$\sigma_1,\dots, \sigma_m
  \colon \partial \Delta^2 \longrightarrow Z$ of the components
  of~$\partial Z$ into~$Z$ are $\pi_1$-injective.  

  Applying Proposition~\ref{prop:relsvfill} (using~$S^1 \cong \partial
  \Delta^2$), we obtain
  \[ \|\beta\|_1 = \sfill_Z(\sigma_1 + \dots +\sigma_m).
  \]
  In combination with Remark~\ref{rem:fillcal} and Lemma~\ref{lem:sclfill},
  this shows that
  \[ \|\beta\|_1 = \sfill_{F(S)}(r_1 + \dots + r_m) = 4 \cdot \scl_S(r_1 + \dots + r_m).
  \qedhere
  \]
\end{proof}

\begin{corr}[$\scl$ as relative $l^1$-semi-norm; free groups]\label{cor:cylrelfree}
  Let $S$ be a set, let $m \in \N$, let $r_1,\dots, r_m \in F(S)'$ be non-trivial,
  let
  \begin{align*}
    Z
    & := \Bigl(\bigvee_S S^1\Bigr) \cup_{r_1,\dots, r_m} \Bigl( \coprod_m S^1 \times [0,1]\Bigr)
    \\
    & := \Bigl(\bigvee_S S^1 \Bigr) \cup_{\coprod_{j=1}^m \gamma_j \text{ on~$\coprod_m S^1 \times\{0\}$}}
         \Bigl( \coprod_m S^1 \times [0,1] \Bigr)
  \end{align*}
  be the mapping cylinder associated with (loops~$\gamma_1,\dots, \gamma_r$ representing)~$r_1,\dots, r_m$, and
  let $\partial Z := \coprod_m S^1 \times \{1\} \subset Z$.
  Moreover, let $\beta \in H_2(Z, \partial Z ;\R)$ be the relative homology class
  whose boundary class~$\partial \beta$ is the fundamental class of~$\coprod_m S^1$.
  Then 
  \[ \|\beta\|_1 = 4 \cdot \scl_S ( r_1 + \dots + r_m).
  \]
\end{corr}
\begin{proof}
  Clearly,  $\bigvee_S S^1$ is a model of~$BF(S)$ and $H_2(F(S);\R) \cong 0$.
  Therefore, we can apply Corollary~\ref{cor:cylrelgen}.
\end{proof}

\subsection{Decomposable relators}\label{subsec:decomprel}

The filling view allows us to compute the $l^1$-semi-norm for
certain classes in degree~$2$ associated to ``decomposable relators''
in terms of stable commutator length. Let us first describe these homology
classes:

\begin{setup}[decomposable relators~I]\label{set:decomprel1}
  Let $G_1$ and $G_2$ be groups that satisfy~$H_2(G_1;\R) \cong 0$ and $H_2(G_2;\R) \cong 0$
  and let $r_1 \in G_1'$, $r_2 \in G_2'$ be elements of infinite order. We then consider the glued group
  \[ D(G_1,G_2,r_1,r_2) :=
     (G_1 * G_2) / \langle r_1 \cdot r_2 \rangle^\triangleleft \cong G_1 *_\Z G_2,
  \]
  where the amalgamation homomorphisms~$\Z \longrightarrow G_1$ and
  $\Z \longrightarrow G_2$ are given by~$r_1$ and~$r_2^{-1}$, respectively.

  Associated with this situation, there is a canonical homology
  class~$\alpha \in H_2(D(G_1,G_2,r_1,r_2);\R)$: Let $X_1$ and
  $X_2$ be classifying spaces for~$G_1$ and $G_2$, respectively.  We
  consider the cylinder spaces
  \begin{align*}
    Z_1 & := X_1 \cup_{r_1 \text{ on~$S^1 \times \{0\}$}}
    \bigl(S^1 \times [0,1]\bigr)
    \\
    Z_2 & := X_2 \cup_{r_2 \text{ on~$S^1 \times \{0\}$}}
    \bigl(S^1 \times [0,1]\bigr)   
  \end{align*}
  for the relators~$r_1$ and $r_2$, respectively. Then
  \[ P := Z_1 \cup_{(z,1) \sim (\overline z,1)} Z_2
  \]
  is a CW-complex such that the canonical maps~$Z_1 \longrightarrow P$
  and $Z_2 \longrightarrow P$ induce an isomorphism~$\pi_1(P) \cong
  D(G_1,G_2,r_1,r_2) =: G$.

  Let $\beta_1 \in H_2(Z_1, S^1 \times\{1\};\R)$ and $\beta_2 \in
  H_2(Z_2,S^1 \times\{1\};\R)$ be the relative classes whose
  boundaries are fundamental classes of~$S^1 \times \{1\}$
  (corresponding to the relators~$r_1$ and $r_2$, respectively).
  Moreover, let $\widetilde \alpha \in H_2(P;\R)$ be the class
  obtained by glueing~$\beta_1$ and~$\beta_2$. Then, we define $\alpha
  \in H_2(G;\R)$ as the image of~$\widetilde \alpha$ under the
  classifying map~$P \longrightarrow BG$ (which is induced by the
  canonical maps~$Z_1 \longrightarrow P$ and $Z_2 \longrightarrow P$). 
\end{setup}

\begin{rmk}[integrality of the canonical class]\label{rem:decomprel1integral}
  In the situation of Setup~\ref{set:decomprel1}, the canonical
  homology class~$\alpha$ is integral: It suffices to show that
  $\widetilde \alpha \in H_2(P;\R)$ is integral. 
  Comparing the long exact sequences of~$(Z_1,\partial Z_1)$
  with $\Z$- and $\R$-co\-efficients shows that $\beta_1 \in H_2(Z_1,\partial Z_1;\R)$
  is an integral class. Analogously, $\beta_2$ is integral. Thus,
  also the glued class~$\widetilde \alpha$ is integral.
\end{rmk}

\begin{setup}[decomposable relators~II]\label{set:decomprel2}
  Let $G_1$ be a group with~$H_2(G_1;\R) \cong 0$ and let $r_1, r_2 \in G_1'$
  be elements of infinite order.
  We then consider the group
  \[ T(G_1,r_1,r_2) := \bigl(G_1 * \langle t \rangle\bigr) / \langle r_1 \cdot t \cdot r_2 \cdot t^{-1}\rangle^{\triangleleft}, 
  \]
  where $t$ is a fresh generator of~$\langle t \rangle \cong \Z$.

  Also here, there is a canonical homology class~$\alpha
  \in H_2(T(G_1,r_1,r_2);\R)$, which is defined as follows: 
  Let $X_1$ be a model of~$BG_1$ and let
  \[ Z := X_1 \cup_{r_1, r_2} \bigl( S^1 \times [0,1] \sqcup S^1 \times [0,1]\bigr)
  \]
  be the cylinder space associated with~$r_1$ and~$r_2$. Let $\beta
  \in H_2(Z,\partial Z;\R)$ be the relative ``fundamental'' class as
  in Corollary~\ref{cor:cylrelgen}.  Glueing the two cylindrical ends
  of~$Z$ by an orientation reversing homeomorphism leads to a
  CW-complex~$P$ such that~$\pi_1(P) \cong T(G_1,r_1,r_2) =: G$ in the
  obvious way (the additional generator~$t$ corresponds to the
  loop~$\{1\} \times ([0,1] \sqcup_{s \sim s} [0,1])$ in the looped
  cylinder.  Let $\widetilde \alpha \in H_2(P;\R)$ be the class
  obtained by glueing~$\beta$ to itself via the cylinder. Then
  we define~$\alpha \in H_2(G;\R)$ as the image of~$\widetilde \alpha$
  under the classifying map~$P \longrightarrow BG$ (which is induced
  by the canonical map~$X_1 \longrightarrow P$ and the cylinder loop).
\end{setup}

\begin{thm}[decomposable relators]\label{thm:decomprelgen}
  Let $G_1$ be a group with~$H_2(G_1;\R) \cong 0$ and let $r_1 \in G_1'$
  be an element of infinite order.
  \begin{enumerate}
  \item \label{item:setup1 relators}
    Let $G_2$ also be a group with~$H_2(G_2;\R) \cong 0$, let $r_2 \in G_2'$
    be an element of infinite order, and let $\alpha \in H_2(D(G_1,G_2,r_1,r_2);\R)$
    be the canonical homology class (Setup~\ref{set:decomprel1}). Then
    \[ \|\alpha\|_1 = 4 \cdot (\scl_{G_1} r_1 + \scl_{G_2} r_2 )
                    = 4 \cdot \Bigl(\scl_{G_1 * G_2} (r_1 \cdot r_2) - \frac12 \Bigr).
    \]
  \item
    Let $r_2 \in G_1'$ be an element of infinite order, let $\alpha
    \in H_2(T(G_1,r_1,r_2);\R)$ be the canonical class (Setup~\ref{set:decomprel2}),
    and let $t$ be the fresh letter in~$T(G_1,r_1,r_2)$. Then
    \[ \|\alpha\|_1 = 4 \cdot \scl_{G_1} (r_1 + r_2)
    = 4 \cdot \Bigl( \scl_{G_1 * \langle t \rangle} (r_1 \cdot t \cdot r_2 \cdot t^{-1})
    - \frac12\Bigr).
    \]
  \end{enumerate}
\end{thm}

\begin{proof}
  In both cases, we will use that the stable commutator length and the
  canonical CW-complexes of decomposable relators can be expressed in
  terms of the stable commutator lengths and cylinder complexes of the
  sub-relators.

  \emph{Ad~1.}
  It is known that~\cite[Proposition~2.99]{Calegari}
  \[ \scl_{G_1* G_2} (r_1 \cdot r_2) = \scl_{G_1} r_1 + \scl_{G_2} r_2 + \frac12.
  \]
  We will now show that $\|\alpha\|_1$ equals~$4 \cdot (\scl_{G_1} r_1 + \scl_{G_2} r_2)$. In
  the following, we will use the notation from Setup~\ref{set:decomprel1}.
  By construction, we have
  \[ \alpha = H_2(c;\R) (\widetilde \alpha),
  \]
  where $c \colon P \longrightarrow BD(G_1,G_2,r_1,r_2)$ is the
  classifying map of~$P$. The mapping
  theorem (Corollary~\ref{cor:l1mappingtheorem}) shows that
  $\|\alpha\|_1 = \| \widetilde \alpha\|_1$. 
  Therefore, it suffices to compute~$\|\widetilde \alpha\|_1$. 
  
  Because $r_1$ and $r_2$ have infinite order, the inclusions of~$S^1
  \times \{1\}$ into~$Z_1$ and~$Z_2$, respectively, are
  $\pi_1$-injective. As $\pi_1(S^1 \times \{1\}) \cong \Z$ is
  amenable, the amenable glueing
  theorem~\cite[Section~6]{bbfipp}
  shows that
  \[ \| \widetilde \alpha\|_1 = \|\beta_1\|_1 + \|\beta_2\|_1;
  \]
  the proofs of Bucher et al.\ carry over from the manifold
  case to this setting, because they established the necessary
  tools in bounded cohomology in this full generality. 
  Moreover, we know that (Corollary~\ref{cor:cylrelgen})
  \[ \|\beta_1\|_1 = 4 \cdot \scl_{G_1} r_1
     \qand
     \|\beta_2\|_1 = 4 \cdot \scl_{G_2} r_2.
  \]
  Putting it all together, we obtain
  $\|\alpha\|_1 = \|\widetilde \alpha\|_1 = 4 \cdot (\scl_{G_1}r_1 + \scl_{G_2} r_2)
  $,  
  as claimed.
  
  \emph{Ad~2.}
  We argue in a similar way as in the first part: 
  In this situation, it is known that~\cite[Theorem~2.101]{Calegari}
  \[ \scl_{G_1 * \langle t \rangle}(r_1 \cdot t \cdot r_2 \cdot t^{-1})
     = \scl_{G_1}(r_1 + r_2) + \frac12.
  \]
  We will now use the notation from Setup~\ref{set:decomprel2}. 
  The 
  classifying map~$c \colon P \longrightarrow BT(G_1,r_1,r_2)$ maps~$\widetilde
  \alpha$ to the canonical class~$\alpha \in H_2(G;\R)$ and the
  mapping theorem (Corollary~\ref{cor:l1mappingtheorem}) shows that $\| \alpha\|_1 =
  \|\widetilde \alpha\|_1$.
  
  Because $r_1$ and $r_2$ have infinite order, we can again use the amenable
  glueing theorem to deduce that
  \[ \|\widetilde \alpha\|_1 = \|\beta\|_1.
  \]
  Moreover, Corollary~\ref{cor:cylrelgen} shows that
  \[ \|\beta\|_1 = 4 \cdot \scl_{G_1} (r_1 + r_2).
  \]
  Therefore, we obtain~$\|\alpha\|_1 = \|\widetilde \alpha_1\|_1 = 4
  \cdot \scl_{G_1}(r_1 + r_2)$.
\end{proof}

In the case that $G_1$ and $G_2$ are free groups, statements of this
type are also contained in arguments of
Calegari~\cite[p.~2004]{calegari_surffromhom}. 

Unfortunately, there does not seem to be an easy way to remove the
condition~$H_2(G_1;\R) \cong 0$ in Theorem~\ref{thm:decomprelgen}:
Without this condition, we have too much ambiguity in the proof of
Proposition~\ref{prop:relsvfill} to ensure integrality and norm control
simultaneously. 

\subsection{Simplicial volume as filling invariant}\label{subsec:svfill}

We mention that also simplicial volume
of higher-dimensional manifolds admits a description as a filling
invariant (this result will not be used in the rest of the paper):

\begin{thm}[simplicial volume as filling invariant]\label{thm:fillsv}
  Let $d \in \N_{\geq 2}$, let $M$ be an orientable closed connected $d$-manifold, let $\tau
  \colon \Delta^d \longrightarrow M$ be an embedding of the standard
  $d$-simplex into~$M$ (i.e., $\tau$ is a homeomorphism onto its image),
  let $\Delta := \tau(\Delta^d) \subset M$, and
  let $\sigma := \tau|_{\partial \Delta^d} \colon \partial \Delta^d
  \longrightarrow M$.
  \begin{enumerate}
  \item Then $\| M \setminus \Delta^\circ, \partial \Delta\| = \sfill_{M \setminus \Delta^\circ} \sigma$.
  \item If $d \geq 3$, then
    $\|M\| = \sfill_{M \setminus \Delta^\circ} \sigma$.
  \item If $d = 2$ and $M \not\simeq S^2$, then
    $\|M\| = \sfill_{M \setminus \Delta^\circ} \sigma - 2$. 
  \end{enumerate}
\end{thm}

\begin{proof}
  \emph{Ad~1.}
  This is a special case of Proposition~\ref{prop:relsvfill}:
  We consider~$Z := M \setminus \Delta^\circ$. 
  The map~$\sigma \colon \partial \Delta^d
  \longrightarrow M \setminus \Delta^\circ = Z$ is $\pi_1$-injective (if
  $d \geq 3$, then $\pi_1(\partial \Delta^d)$ is trivial; if $d = 2$,
  this holds by the classification of surfaces and the assumption
  $M \not\simeq S^2$). In view of Poincar\'e duality, the
  hypothesis on~$H_d(Z,\partial Z;\R)$ is satisfied. We therefore
  can apply Proposition~\ref{prop:relsvfill} to obtain
  \[ \| Z, \partial Z\|
  = \bigl\| [Z,\partial Z]_\R \bigr\|_1
  = \sfill_Z \sigma.
  \]
  
  \emph{Ad~2.}  In view of the first part, it suffices to show
  that~$\|M\| = \| M \setminus \Delta^\circ, \partial \Delta\|$. 
  The
  amenable glueing theorem for simplicial
  volume~\cite{vbc,bbfipp,frigeriomoraschini}  shows that
  \[ \|M\| = \|M \setminus \Delta^\circ, \partial \Delta\| + \|\Delta^d, \partial \Delta^d\|
  \]
  (because~$d\geq 3$, both inclusions~$\partial \Delta \longrightarrow M \setminus \Delta^\circ$
  and $\partial \Delta \longrightarrow \Delta$ are $\pi_1$-injective).  
  Moreover, $\|\Delta^d, \partial \Delta^d\| = 0$~\cite{vbc}. Hence, 
  we obtain
  $\|M\| = \|M\setminus \Delta^\circ, \partial \Delta\|$.
    
  \emph{Ad~3.}
  Again, 
  in view of the first part, it suffices to show
  that~$\|M\| = \| M \setminus \Delta^\circ, \partial \Delta\| - 2$.
  In this, two-dimensional, case, this equality follows from the
  classification of compact surfaces and the computation
  of the (relative) simplicial volume of compact surfaces in terms of their
  genus~\cite{vbc,thurston}.
\end{proof}

\subsection{Proof of Theorems \ref{theorem:doubling} and \ref{theorem:nogapgroup}}

Theorem \ref{theorem:doubling} is a special case of Theorem \ref{thm:decomprelgen} with the decomposable relators of Setup \ref{set:decomprel1}.
Let $G$ be a group that satisfies $H_2(G;\R) \cong 0$ and let $r \in G'$ be an element of infinite order. Then we define the \emph{double} $D(G,r)$ of $G$ and $r$ by setting
$$
D(G,r) := D(G,G,r,r) = (G_{\text{left}} \star G_{\text{right}})/ \langle r_{\text{left}} \cdot r_{\text{right}}\rangle^\triangleleft,
$$
where $G_{\text{left}}$ and $G_{\text{right}}$ are isomorphic copies of $G$ and $r_{\text{right}} \in G_{\text{right}}$, $r_{\text{left}} \in G_{\text{left}}$ are the elements corresponding to $r \in G$.
Observe that if $G$ is finitely presented, then so is $D(G,r)$. As in Setup~\ref{set:decomprel1} there is a canonical integral class $\alpha \in H_2(D(G,r);\R)$.

\begin{corr}[Theorem \ref{theorem:doubling}]\label{cor:double}
  Let $G$ be a group with~$H_2(G;\R) \cong 0$ and let $r \in G'$ be 
  of infinite order. Then the canonical integral class~$\alpha \in H_2(D(G,r);\R)$
  satisfies
  \[ \| \alpha\|_1 = 8 \cdot \scl_G r.
  \]
\end{corr}
\begin{proof}
  This is an immediate corollary of Theorem~\ref{thm:decomprelgen} (\ref{item:setup1 relators}).
\end{proof}

Applying Theorem \ref{theorem:doubling} to the universal central extension $E$ of Thompson's group $T$, we deduce Theorem~\ref{theorem:nogapgroup}:

\begin{corr}[Theorem~\ref{theorem:nogapgroup}] \label{cor:controlled integral 2 classes}
For every $q \in \Q_{\geq 0}$, there is a finitely presented group~$G_q$ and an integral class~$\alpha_q \in H_2(G_q;\R)$ such that $\| \alpha_q \|_1 = q$. In particular, for every $\epsilon > 0$ there is a finitely presented group $G_\epsilon$ and an integral class~$\alpha_\epsilon \in H_2(G_\epsilon;\R)$ such that $0 < \| \alpha_\epsilon \|_1 \leq \epsilon$.
\end{corr}

\begin{proof}
  Let $q \in \Q_{\geq 0}$. For~$q=0$ we can take the zero class of the trivial group 
  (or any integral $2$-class in any finitely presented amenable group).
  
  For $q > 0$, let 
 $r_q \in E$ be an element in the universal central extension $E$ of Thompson's group $T$  with~$\scl_E r_q = q/8$.  
 Proposition~\ref{prop:scl on central extension of thompsons group} asserts that such an element exists, that $E$ is finitely presented and that $H_2(E;\Z)\cong 0$ and hence $H_2(E;\R)\cong0$. 
 As $\scl_E r_q > 0$ the element $r_q \in E$ has infinite order. 
 
 Let $G_q := D(E,r_q)$ be the double and let $\alpha_q \in H_2(G_q;\R)$ be the associated integral $2$-class.
 Theorem~\ref{theorem:doubling} shows that
 $$
 \| \alpha_q \|_1 = 8 \cdot \scl_E r_q = q.
 $$

 We note that one can prove the second part of
 Theorem~\ref{theorem:nogapgroup} also via previously known examples of
 stable commutator length (Example~\ref{exa:hMCG}).
\end{proof}

\section{The $l^1$-semi-norm of products with surfaces: Proof of Theorem \ref{theorem:prodnorm}} \label{sec:l1 norm of prod with surfaces}

Bucher~\cite{bucherprodsurf} computed the simplicial volume of the product of two surfaces
(see Theorem~\ref{thm:bucher_prod_surf}). We will use her techniques to 
generalise this statement to the product of more general
$2$-classes. 
This will allow us to construct integral $4$-classes whose $l^1$-semi-norm can be expressed in terms of the $l^1$-semi-norm of $2$-classes.
Theorem~\ref{theorem:prodnorm} will be a corollary (Corollary \ref{corr:exact 4-classes as product with surface}) of these constructions.

\begin{thm}\label{thm:prodnorm}
  Let $G$ and $\Gamma$ be groups, let $\alpha \in H_2(G;\R)$, and
  $\beta \in H_2(\Gamma;\R)$.  Furthermore, let $\rho \colon \Gamma
  \longrightarrow \Homeo(S^1)$ be a circle action of~$\Gamma$. Assume furthermore that~$\rho^* \Or \in C^2_b(\Gamma;\R)$ is an extremal
  cocycle for~$\beta$; see Section~\ref{subsec:euler class}. Then
  the class~$\alpha \times \beta \in H_4(G \times \Gamma;\R)$
  satisfies
  \[ \| \alpha \times \beta\|_1
     = \frac 32 \cdot \|\alpha \|_1 \cdot \| \beta \|_1.
  \]
\end{thm}

Theorem~\ref{thm:prodnorm} is a strict
generalisation of Bucher's result~\cite{bucherprodsurf} and our proof follows the outline of Bucher's work. Recall that for $g \geq 2$ we denote the oriented closed connected
surface of genus~$g$ by~$\Sigma_g$, its fundamental group by~$\Gamma_g$, 
and its fundamental class by~$[\Sigma_g]_\R \in H_2(\Sigma_g;\R) \cong H_2(\Gamma_g;\R)$.
Fix a hyperbolic structure on~$\Sigma_g$ and let $\rho \col \Gamma_g \to \Homeo(S^1)$ be the corresponding action on the boundary $\partial \Gamma_g \cong S^1$.
 The cocycle~$\rho^*\Or \in C^2_b(\Gamma_g;\R)$ is extremal to~$[\Sigma_g]_\R$ (see Section~\ref{subsec:euler class}) and satisfies
 $$
 \bigl\langle [\rho^* \Or], [\Sigma_g]_\R \bigr\rangle = \| \Sigma_g \| = 4 \cdot g -4.
 $$
 Therefore, we obtain the following immediate corollary to Theorem \ref{thm:prodnorm}:
 \begin{corr}[Theorem \ref{theorem:prodnorm}] \label{corr:exact 4-classes as product with surface}\label{cor:prodnormsurf}
   Let $g \geq 2$, 
   let $G$ be a group, let $\alpha \in H_2(G;\R)$ and let $\Gamma_g$ and $[\Sigma_g]_\R \in H_2(\Gamma_g;\R)$ be as above.
   Then the class~$\alpha \times [\Sigma_g]_\R \in H_4(G \times \Gamma_g;\R)$ satisfies
   $$
   \bigl\| \alpha \times [\Sigma_g]_\R \bigr\|_1 = \frac{3}{2} \cdot \| \alpha \|_1 \cdot \| \Sigma_g \| =  6 \cdot (g - 1) \cdot \| \alpha \|_1.
   $$
\end{corr}

\begin{proof}[Proof of Theorem \ref{thm:prodnorm}]
In this proof, all cocycles will be given in the homogeneous resolution.
The upper bound holds for all classes in degree~$2$
(Corollary~\ref{cor:l1prod2improvedupperbound}).  For the lower bound
we will use duality (Proposition~\ref{prop:duality simvol bc}):

Let $\orb \R \Gamma := [\rho^*\Or]^\R = -2 \cdot \eub \R\Gamma \in H^2_b(\Gamma; \R)$ be the orientation class for the given
action~$\rho \col \Gamma \to \Homeo(S^1)$. 
 Moreover, let $\omega \in C^2_b(G;\R)$ be an extremal cocycle for~$\alpha \in H_2(G;\R)$ in the homogeneous resolution; see Proposition \ref{prop:duality simvol bc}.
By possibly replacing~$\omega$ by~$\alt^2_b(\omega)$, we may assume that $\omega$ is alternating; see Section~\ref{subsec:alternating}. 
By assumption, $\| \omega \|_\infty \leq 1$ and 
$$
\bigl\langle [\omega], \alpha \bigr\rangle = \| \alpha \|_1 \qand \langle \orb \R \Gamma, \beta \rangle = \| \beta \|_1.
$$
The cross-product  $ \omega \times \rho^*\Or \in C^4_b(G \times \Gamma; \R)$ of $\omega$ and $\rho^*\Or$ is defined via
$$
\omega \times \rho^*\Or \col \bigl((g_1, \gamma_1), \dots, (g_5,\gamma_5)\bigr)
\longmapsto \omega(g_1,g_2,g_3) \cdot \rho^*\Or{}{}(\gamma_3,\gamma_4,\gamma_5)
$$
and satisfies 
$$
\bigl\langle  [\omega \times \rho^*\Or{}{}], \alpha \times \beta \bigr\rangle
= \| \alpha \|_1 \cdot \|\beta \|_1.
$$
This recovers the estimate $\| \alpha \|_1 \cdot \| \beta \|_1 \leq \| \alpha \times \beta \|_1$ as seen in Proposition~\ref{prop:l1productgeneric}.

\begin{claim}  \label{claim:norm of theta 2/3}
Let $\Theta := \alt^4_b(\omega \times \rho^*\Or)$ 
be the associated alternating cocycle of~$\omega \times \rho^*\Or$; see Section~\ref{subsec:alternating}. Then 
$\| \Theta \|_\infty \leq 2/3$.
\end{claim}

Once Claim~\ref{claim:norm of theta 2/3} is established, we can argue as follows: 
Recall that $\Theta$ and $\omega \times \rho^*\Or$ represent the same class in $H^4_b(G \times \Gamma;\R)$ by Proposition \ref{prop:alternating cocycles}.
Hence, $\langle \Theta, \alpha \times \beta \rangle = \| \alpha \|_1 \cdot \| \beta \|_1$.
Moreover by the claim we have that $\| \frac{3}{2} \cdot \Theta \|_\infty \leq 1$ and by duality we conclude that
$\frac{3}{2} \cdot \| \alpha \|_1 \cdot \| \beta \|_1 \leq \| \alpha \times \beta \|_1$.
Putting both estimates together, we will obtain
$$\frac{3}{2} \cdot \| \alpha \|_1 \cdot \| \beta \|_1 = \| \alpha \times \beta \|_1,
$$ as claimed in Theorem \ref{thm:prodnorm}.
To complete the proof, it thus only remains to show Claim~\ref{claim:norm of theta 2/3}.
\end{proof}

\begin{proof}[Proof of Claim \ref{claim:norm of theta 2/3}]
  We will follow the outline of Bucher's proof~\cite[Proposition~7]{bucherprodsurf},
  quoting parts of the proof verbatim. 
Let $g_0, \ldots, g_4 \in G$ and $\gamma_0, \ldots, \gamma_4 \in \Gamma$.
Moreover, let $\xi \in S^1$ be a point to define~$\Or$ and set $x_i := \rho(\gamma_i).\xi$
for all~$i \in \{0, \ldots, 4 \}$.
We will give upper bounds to $|\Theta((g_0, \gamma_0), \ldots, (g_4, \gamma_4))|$ in different cases.
By construction,
$$
\Theta\bigl((g_0,\gamma_0), \ldots, (g_4,\gamma_4)\bigr)
$$ 
may be written as
\begin{eqnarray} \label{equn:Theta definition}
\frac{1}{| S_5 |} \cdot \sum_{\sigma \in S_5} \sign(\sigma) \cdot \omega(g_{\sigma(0)}, g_{\sigma(1)}, g_{\sigma(2)}) \cdot \Or(x_{\sigma(2)},x_{\sigma(3)}
,x_{\sigma(4)} ),
\end{eqnarray}
where $\Or$ is the orientation; see Section~\ref{subsec:euler class}.
Every permutation~$\sigma \in S_5$ may be written uniquely as $\sigma = (\mbox{0 1 2 3 4})^k \circ \upsilon$, where $k \in \{0,\ldots,4 \}$  and $\upsilon \in S_5$ is a permutation with~$\upsilon(2)=0$.
Using that both $\Or$ and $\omega$ are alternating we obtain that 
\begin{eqnarray} \label{equn:theta with a(k)}
\Theta\bigl((g_0,\gamma_0), \ldots, (g_4,\gamma_4)\bigr) = \frac{1}{30} \cdot 
\sum_{k=0}^4 A(k),
\end{eqnarray}
where
\begin{align*}
  A(k) = 
  \;& 
  \omega(g_{\tau(0)}, g_{\tau(1)}, g_{\tau(2)})
  \cdot\Or(x_{\tau(0)}, x_{\tau(3)}, x_{\tau(4)}) \\
+ \;&  
\omega(g_{\tau(0)}, g_{\tau(3)}, g_{\tau(4)})
\cdot\Or(x_{\tau(0)}, x_{\tau(1)}, x_{\tau(2)}) \\ 
- \;&
\omega(g_{\tau(0)}, g_{\tau(1)}, g_{\tau(3)})
\cdot\Or(x_{\tau(0)}, x_{\tau(2)}, x_{\tau(4)}) \\ 
- \;&
\omega(g_{\tau(0)}, g_{\tau(2)}, g_{\tau(4)})
\cdot\Or(x_{\tau(0)}, x_{\tau(1)}, x_{\tau(3)}) \\ 
+ \;&
\omega(g_{\tau(0)}, g_{\tau(1)}, g_{\tau(4)})
\cdot\Or(x_{\tau(0)}, x_{\tau(2)}, x_{\tau(3)}) \\ 
+ \;&
\omega(g_{\tau(0)}, g_{\tau(2)}, g_{\tau(3)})
\cdot\Or(x_{\tau(0)}, x_{\tau(1)}, x_{\tau(4)}) 
\end{align*}
for $\tau = (\mbox{0 1 2 3 4})^k$.
Observe also that we may assume that the $x_i$ are cyclically ordered, as $\Theta$ is alternating.

In what follows we will estimate the terms~$A(k)$, depending on the relative position of the~$x_i$: 
\begin{itemize}
\item All $x_0, \ldots, x_4$ are distinct. 
As all $\Or$-terms in $A(k)$ equal~$1$ we have
\begin{align*}
  A(k) = 
  \;& 
\omega(g_{\tau(0)}, g_{\tau(1)}, g_{\tau(2)})+
\omega(g_{\tau(0)}, g_{\tau(3)}, g_{\tau(4)}) \\ 
- \;&
\omega(g_{\tau(0)}, g_{\tau(1)}, g_{\tau(3)})
-
\omega(g_{\tau(0)}, g_{\tau(2)}, g_{\tau(4)}) \\ 
+ \; &
\omega(g_{\tau(0)}, g_{\tau(1)}, g_{\tau(4)})
 +
\omega(g_{\tau(0)}, g_{\tau(2)}, g_{\tau(3)})
 \\
 = 
 \; & 
\omega(g_{\tau(2)}, g_{\tau(3)}, g_{\tau(4)})+
\omega(g_{\tau(0)}, g_{\tau(1)}, g_{\tau(2)}) \\ 
- \; &
\omega(g_{\tau(0)}, g_{\tau(1)}, g_{\tau(3)})
+
\omega(g_{\tau(0)}, g_{\tau(1)}, g_{\tau(4)}) 
\end{align*}
for $\tau = (\mbox{0 1 2 3 4})^k$, where in the last equation we used that
\begin{align*}
0 = & \;\delta^2 \omega(g_{\tau(0)}, g_{\tau(2)}, g_{\tau(3)}, g_{\tau(4)}) \\
= &\; \omega(g_{\tau(2)}, g_{\tau(3)}, g_{\tau(4)}) - \omega(g_{\tau(0)}, g_{\tau(3)}, g_{\tau(4)}) \\
 + &\; \omega(g_{\tau(0)}, g_{\tau(2)}, g_{\tau(4)}) - \omega(g_{\tau(0)}, g_{\tau(2)}, g_{\tau(3)})
\end{align*}
by the cocycle condition.
In particular, we see that $|A(k)| \leq 4$ as $\| \omega \|_\infty \leq 1$.
Hence,
$$
\bigl| \Theta((g_1,\gamma_1), \ldots, (g_5,\gamma_5)) \bigr|
\leq \frac{1}{30} \cdot \sum_{k=0}^4 |A(k)| \leq \frac{20}{30} = \frac{2}{3}.
$$

\item Two of the~$x_i$ are identical and the others are distinct.
Without loss of generality assume that~$x_0=x_1$. 
Observe that in this case $\Or(x_i,x_j,x_k)=0$ whenever two of the~$x_i, x_j, x_k$ are equal to $x_0$ or~$x_1$.
We will estimate~$|A(k)|$ in different cases:
\begin{itemize}
\item $k=0$: Then $A(0) = 
\omega(g_0, g_1, g_2) -
\omega(g_{0}, g_{1}, g_{3}) +
\omega(g_{0}, g_{1}, g_{4})$
and hence $|A(0)|\leq 3$.

\item $k=1$: Then $A(1) =  
\omega(g_1, g_4, g_0) 
-\omega(g_1, g_3, g_0) 
+\omega(g_1, g_2, g_0)$
and hence $|A(1)|\leq 3$.

\item $k=2$: Then
\begin{align*}
A(2) = &\; 
\omega(g_2, g_0, g_1) 
-\omega(g_{2}, g_{3}, g_{0})
-\omega(g_{2}, g_{4}, g_{1})
\\
+&\; \omega(g_{2}, g_{3}, g_{1})
+\omega(g_{2}, g_{4}, g_{0}).
\end{align*}
By the cocycle condition, it follows that 
\begin{align*}
0 &= \delta^2 \omega(g_2,g_4,g_0,g_1) \\
&= \omega(g_4,g_0,g_1)-\omega(g_2,g_0,g_1)+\omega(g_2,g_4,g_1)-\omega(g_2,g_4,g_0).
\end{align*}
Therefore, $A(2) =  \omega(g_4,g_0,g_1) 
-\omega(g_{2}, g_{3}, g_{0})
+\omega(g_{2}, g_{3}, g_{1})
$
and so $|A(2)|\leq 3$.

\item $k=3$: 
Then
\begin{align*}
A(3) = &\; 
\omega(g_3, g_4, g_0)
+\omega(g_3, g_1, g_2) 
-\omega(g_3, g_4, g_1)
\\
-&\;\omega(g_3, g_0, g_2)
+\omega(g_3, g_0, g_1)
\end{align*}
and hence $|A(3)|\leq 5$.
\item $k=4$: 
Then
\begin{align*}
A(4) = &\;
\omega(g_4, g_0, g_1)
-\omega(g_4, g_0, g_2)
-\omega(g_4, g_1, g_3)
\\
& \;
+\omega(g_4, g_0, g_3)
+\omega(g_4, g_1, g_2)
\end{align*}
and hence $|A(4)| \leq 5$.
\end{itemize}

Putting things together we see that
\begin{eqnarray*}
\bigl|\Theta((g_1,\gamma_1), \ldots, (g_5,\gamma_5))\big| \leq
\frac{1}{30}
\sum_{k=0}^4 |A(k)| \leq \frac{19}{30} < 2/3.
\end{eqnarray*}

\item Three of the $x_i$ are identical, the other ones are different. As $\Theta$ is alternating we may assume that $x_0=x_1=x_2$.
A permutation $\sigma \in S_5$ for which 
the
$\Or(x_{\sigma(2)}, x_{\sigma(3)}, x_{\sigma(4)})$
term in Equation (\ref{equn:Theta definition}) is non-trivial has to map exactly one of the elements in $\{2,3,4 \}$ to one of the elements $\{0,1,2 \}$, and has to map the remaining two elements of $\{2,3,4\}$ to~$\{3,4\}$. We then have two more choices for $\sigma(0)$ and~$\sigma(1)$.
We compute that the total number of such permutations is~$36$. For all other permutations, the $\Or$-term in Equation (\ref{equn:Theta definition}) will vanish.
We may then estimate
$$
\bigl|\Theta((g_0,\gamma_0), \ldots, (g_4,\gamma_4))\bigr|
\leq \frac{36}{5!} = \frac{36}{120} < \frac{2}{3}.
$$

\item Two pairs are identical and one element is different from these pairs. Assume without loss of generality that $x_0=x_1$, $x_2=x_3$. A permutation $\sigma \in S_5$ for which
the $\Or(x_{\sigma(2)}, x_{\sigma(3)}, x_{\sigma(4)})$-term in Equation (\ref{equn:Theta definition}) is non-trivial has to map
each of $\{2,3,4 \}$ to different sets $\{0,1 \}$, $\{ 2,3 \}$, and $\{5 \}$. Moreover, there are two choices for the two elements that get mapped to the sets with two elements.
Again, there are two more choices for $\sigma(0)$ and $\sigma(1)$.
We compute that there are a total of $48$ such permutations. For all other permutations the $\Or$-term in Equation (\ref{equn:Theta definition}) will vanish.
We may then estimate
$$
\bigl|\Theta((g_0,\gamma_0), \ldots, (g_4,\gamma_4))\bigr| \leq \frac{48}{5!} = \frac{48}{120}< \frac{2}{3}.
$$

\item If more than three of the $x_i$ are identical or if exactly three of the $x_i$ are identical and the two remaining $x_i$ are identical, then the $\Or$-term in Equation (\ref{equn:Theta definition}) always vanishes and we get that
$$
\bigl|\Theta((g_0,\gamma_0), \ldots, (g_4,\gamma_4))\bigr| = 0 < \frac{2}{3}.
$$
\end{itemize}

In summary, in each case we have seen that 
$$
\bigl| \Theta((g_0,\gamma_0), \ldots, (g_4,\gamma_4)) \bigr| \leq \frac{2}{3}
$$
and hence $\| \Theta \|_\infty \leq \frac{2}{3}$.
This finishes the proof of Claim \ref{claim:norm of theta 2/3}
(and also the proof of Theorem~\ref{thm:prodnorm}).
\end{proof}

\section{Manufacturing manifolds with controlled simplicial volumes}
\label{sec:manufac sim vol}

The computation of $\ell^1$-semi-norms of
$2$-classes in group homology allows us to construct manifolds with controlled simplicial volume. 

This construction will involve a normed version of Thom's realisation theorem, which we recall in Section~\ref{subsec:thom realisation}.
Theorem~\ref{theorem:nogap} is proven in Section~\ref{subsec:proofA} and the theorems for dimension~$4$ (Theorems~\ref{theorem:simvolQ} and~\ref{theorem:exact4mfd}) are proven in Section~\ref{subsec:proofB}.
Finally, in Section~\ref{subsec:related problems} we discuss related problems and further research topics.

\subsection{Thom's realisation theorem}\label{subsec:thom realisation}

In order to turn classes in group homology into manifolds with controlled
simplicial volume, we will use the following \emph{normed} version of Thom's
realisation theorem:

\begin{thm}[normed Thom realisation]\label{thm:thom}
  For each~$d \in \N_{\geq 4}$, there exists a constant~$K_d \in \N_{>0}$ with the following property:
  If $G$ is a finitely presented group (with model~$X$ of~$BG$) and $\alpha \in H_d(X;\R)$
  is an integral homology class, then there is an oriented closed connected (smooth) $d$-manifold~$M$,
  a continuous map~$f \colon M \longrightarrow X$ and a number~$m \in \{1,\dots, K_d\}$
  with
  \[ H_d(f ;\R) [M]_\R = m \cdot \alpha
     \qand \|M\| = m \cdot \|\alpha\|_1.
  \]
  Moreover, one can choose~$K_4 = 1$ and $K_5 = 1$.
\end{thm}

\begin{proof}
  Everything except for the condition on the simplicial volume is contained in
  Thom's classical realisation theorems~\cite[Theorems~III.3,~III.4]{thomtoplib}.
  (Thom's original theorems apply to~$X$ because every singular homology class of~$X$
  is supported on a finite subcomplex; as $G$ is finitely presented,
  we can choose the subcomplex in such a way that the inclusion into~$X$ induces
  a $\pi_1$-isomorphism.)
  One can then apply surgery to obtain a manifold representation of~$\alpha$,
  where $f \colon M \longrightarrow X$ in addition is a $\pi_1$-isomorphism~\cite[(proof of)
  Theorem~3.1]{crowleyloeh} (this will not touch the multiplier~$m$). Therefore,
  the mapping theorem for the $l^1$-semi-norm (Corollary~\ref{cor:l1mappingtheorem})
  shows that
  \[ \| M \| = \bigl\| H_d(f;\R)[M]_\R \bigr\|_1
             = \| m \cdot \alpha \|_1 = m \cdot \| \alpha\|_1.
  \qedhere	     
  \]
\end{proof}

\subsection{No gaps in higher dimensions: Proof of Theorem~\ref{theorem:nogap}}\label{subsec:proofA}

We promote the computations of $l^1$-semi-norms in degree~$2$ to higher dimensions using cross-products. The manifolds will then be provided by the normed Thom realisation (Theorem \ref{thm:thom}).

\begin{proof}[Proof of Theorem~\ref{theorem:nogap}]
  Let $d \in \N_{\geq 4}$. We fix an oriented closed connected
  hyperbolic $(d-2)$-manifold~$N_d$; in particular, $\|N_d\| > 0$.
  Moreover, let $K_d \in \N$ be the constant provided by Thom's
  realisation theorem (Theorem~\ref{thm:thom}).

  Let $\varepsilon \in \R_{>0}$.  By Theorem~\ref{theorem:nogapgroup},
  there exists a finitely presented group~$G$ and an integral
  class~$\alpha \in H_2(G;\R)$ with
  $0 < \|\alpha\|_1 \leq \epsilon.
  $ 
  Let $X$ be a model of~$BG$. Then the product class
  \[ \alpha' := \alpha \times [N_d]_\R \in H_d(X \times N_d;\R)
  \]
  is integral and satisfies (by Proposition~\ref{prop:l1productgeneric})
  $$
  0 < \|\alpha\|_1 \cdot \|N_d\| \leq \| \alpha' \|_1 \leq {d \choose 2} \cdot \|N_d \| \cdot \epsilon.  
  $$
  
  The normed version of Thom's realisation theorem (Theorem~\ref{thm:thom})
  provides an orientable closed connected $d$-manifold~$M$
  and a number~$m \in \{1,\dots, K_d\}$ with
  \[ \|M \| = m \cdot \|\alpha'\|_1.
  \]
  We conclude that
  $$
  0 < \| M \| \leq {d \choose 2} \cdot \|N_d \| \cdot K_d \cdot \epsilon.
  $$
  As the constants on the right hand side just depend on $d$, this shows that there is no gap at $0$ in $\SV(d)$, the set of simplicial volumes of orientable closed connected $d$-manifolds.
  By additivity (Remark~\ref{rem:svsetbasics}), the set~$\SV(d)$ is also dense in~$\R_{\geq 0}$.
\end{proof}

\subsection{Dimension~$4$: Proofs of Theorems \ref{theorem:simvolQ} and \ref{theorem:exact4mfd}}\label{subsec:proofB}

In dimension~$4$, we have more control on the $l^1$-norm of integral
$4$-classes in group homology (Theorem \ref{theorem:prodnorm}).  This
allows us to prove Theorems~\ref{theorem:simvolQ}
and~\ref{theorem:exact4mfd}.

\begin{prop}\label{prop:simvolfromgroups}
  Let $G$ be a finitely presented group and let $\alpha \in H_2(G;\R)$
  be an integral class. Then there exists an orientable closed connected
  $4$-manifold~$M_\alpha$ with
  \[ \| M_\alpha \| = 6 \cdot \|\alpha\|_1.
  \]
\end{prop}
\begin{proof}
  We proceed as in the proof of Theorem~\ref{theorem:nogap} 
  and consider the product class
  \[ \alpha' := \alpha \times [\Sigma_2]_\R \in H_4(G \times \Gamma_2;\R)
  \]
of $\alpha$ with the fundamental class $[\Sigma_2]_{\R}$ of a surface of genus $2$. Observe that $\alpha'$ is also integral. Then the normed Thom realisation (Theorem~\ref{thm:thom})
  shows that there exists an orientable closed connected $4$-manifold~$M_\alpha$
  with~$\| M_\alpha\| = \|\alpha'\|_1$. We now apply the norm computation from
  Corollary~\ref{corr:exact 4-classes as product with surface} and obtain
  \[ \|M_\alpha\| = \| \alpha'\|_1 
           = 6 \cdot \|\alpha\|_1.
	   \qedhere 
  \]
\end{proof}

\begin{proof}[Proof of Theorem~\ref{theorem:simvolQ}]
  We only need to combine Theorem~\ref{theorem:nogapgroup} (which
  allows to realise any non-negative rational number as
  $l^1$-semi-norm of an integral $2$-class of a finitely presented
  group) with Proposition~\ref{prop:simvolfromgroups}.
\end{proof}

Moreover, we can summarise the relation between stable commutator length and simplicial volumes
in dimension~$4$ as follows:

\begin{corr}[Theorem \ref{theorem:exact4mfd}, dimension~$4$, exact values via~$\scl$]\label{cor:exact4}
  Let $G$ be a finitely presented group with~$H_2(G;\R) \cong 0$ and let $g \in G'$ be an element in the commutator subgroup.
  Then there exists an orientable closed connected $4$-manifold~$M_g$ with
  \[ \|M_g\| = 48 \cdot \scl_G g.
  \]
\end{corr}

\begin{proof}
 We may assume without loss of generality that $r$ has infinite order
 (otherwise we can just take~$M = S^4$).  
 We again consider the doubled group~$D(G,r)$ (as in Corollary~\ref{cor:double})
 and the canonical homology class~$\alpha \in H_2(D(G,r);\R)$, which is integral
 (Remark~\ref{rem:decomprel1integral}); as $G$ is finitely presented, also $D(G,r)$
 is finitely presented. Applying Corollary~\ref{cor:double} shows that
 \[ \|\alpha\|_1 = 8 \cdot \scl_G r.
 \]
 In combination with Proposition~\ref{prop:simvolfromgroups}, we therefore obtain
 an orientable closed connected $4$-manifold~$M$ with
 \[ \| M \|  = 6 \cdot \|\alpha\|_1 = 48 \cdot \scl_G r.
 \qedhere
 \]
\end{proof}
 
\begin{rmk}
  The concrete example manifolds in the proof of Theorem~\ref{theorem:simvolQ}, in
  general, might have different fundamental group; however, by
  construction, their first Betti numbers are uniformly bounded: For each~$q\in \Q$,
  we have
  \[ b_1 (M_q;\Q) \leq \rk \pi_1(M_q) = \rk \pi_1\bigl(D(E,e_q) \times \Gamma_2\bigr) \leq 2 \cdot \rk E + 4.
  \]
\end{rmk}

\subsection{Related problems} \label{subsec:related problems}


The techniques of this paper may be adopted to construct $4$-manifolds with transcendental simplicial volume. Theorem~\ref{theorem:exact4mfd} reduces this problem to finding an appropriate finitely presented group with transcendental stable commutator length.
By constructing such groups explicitly, we could show that there are $4$-manifolds with arbitrarily small transcendental simplicial volume~\cite[Theorem~A]{trans_simvol}. Moreover, we could also show that the set of simplicial volumes is contained in the (countable) set~$\mbox{RC}^{\geq 0}$ of non-negative right-computable numbers~\cite[Theorem~B]{trans_simvol}.
It is unkown which real numbers arise as the stable commutator length of elements in the class of finitely presented groups. However, it is known~\cite{rp_scl} that for the class of recursively presented groups this set is exactly~$\mbox{RC}^{\geq 0}$.


Our techniques for manufacturing manifolds with controlled simplicial
volumes are based on group-theoretic methods and not on genuine
manifold-geo\-metric constructions. One might wonder whether
Theorem~\ref{theorem:nogap} also holds under additional topological or
geometric conditions such as asphericity or curvature conditions.


Originally, we set out to study simplicial volume of one-relator
groups and its relation with stable commutator length. However, we
then realised that some of the techniques applied in a much broader
context (with a weak homological condition). We discuss a connection
between the $l^1$-semi-norm of the relator-class in one-relator groups 
with the stable commutator length of the relator in the free group
in a separate article~\cite{heuerloeh_onerel}.

{\small
\bibliographystyle{alpha}
\bibliography{bib_l1}}

\vfill

\noindent
\emph{Nicolaus Heuer}\\[.5em]
  {\small
  \begin{tabular}{@{\qquad}l}
DPMMS, University of Cambridge\\
    \textsf{nh441@cam.ac.uk},
    \textsf{https://www.dpmms.cam.ac.uk/$\sim$nh441/} 
  \end{tabular}}

\medskip

\noindent
\emph{Clara L\"oh}\\[.5em]
  {\small
  \begin{tabular}{@{\qquad}l}
    Fakult\"at f\"ur Mathematik,
    Universit\"at Regensburg,
    93040 Regensburg\\
    \textsf{clara.loeh@mathematik.uni-r.de}, 
    \textsf{http://www.mathematik.uni-r.de/loeh}
  \end{tabular}}

\end{document}